\documentclass[12pt]{article}
\usepackage{amsthm, amsmath, amssymb, enumerate}
\usepackage{fullpage}
\usepackage{stfloats}
\usepackage[colorlinks=true]{hyperref}
\usepackage{tikz}
\usetikzlibrary{positioning}

\usepackage[all]{xy}
\usepackage{esint}

\begin{document}
\title{Modular Heights of Unitary Shimura Varieties II: Arithmetic Generating Series of Divisors}
\author{Ziqi Guo}
\maketitle

\theoremstyle{plain}
\newtheorem{thm}{Theorem}[section]
\newtheorem{theorem}[thm]{Theorem}
\newtheorem{cor}[thm]{Corollary}
\newtheorem{corollary}[thm]{Corollary}
\newtheorem{lem}[thm]{Lemma}
\newtheorem{lemma}[thm]{Lemma}
\newtheorem{pro}[thm]{Proposition}
\newtheorem{proposition}[thm]{Proposition}
\newtheorem{prop}[thm]{Proposition}
\newtheorem{definition}[thm]{Definition}
\newtheorem{assumption}[thm]{Assumption}
\def\avint{\mathop{\,\rlap{-}\!\!\int}\nolimits}

\theoremstyle{remark}
\newtheorem{remark}[thm]{Remark}
\newtheorem{example}[thm]{Example}
\newtheorem{remarks}[thm]{Remarks}
\newtheorem{problem}[thm]{Problem}
\newtheorem{exercise}[thm]{Exercise}
\newtheorem{situation}[thm]{Situation}
\newtheorem{acknowledgment}[thm]{Acknowledgment}

\numberwithin{equation}{subsection}

\newcommand{\ZZ}{\mathbb{Z}}
\newcommand{\CC}{\mathbb{C}}
\newcommand{\QQ}{\mathbb{Q}}
\newcommand{\RR}{\mathbb{R}}
\newcommand{\HH}{\mathcal{H}}     

\newcommand{\ad}{\mathrm{ad}}            
\newcommand{\NT}{\mathrm{NT}}
\newcommand{\nonsplit}{\mathrm{nonsplit}}
\newcommand{\Pet}{\mathrm{Pet}}
\newcommand{\Fal}{\mathrm{Fal}}
\newcommand{\Af}{\mathbb{A}_f}

\newcommand{\cs}{{\mathrm{cs}}}

\newcommand{\XU}{X_U}
\newcommand{\Fn}{F_v}
\newcommand{\LU}{L_U}
\newcommand{\LL}{\overline{\mathcal{L}}}
\newcommand{\OF}{\mathcal{O}_F}
\renewcommand{\OE}{\mathcal{O}_E}
\newcommand{\XXU}{\mathcal{X}_U}
\newcommand{\OA}{\underline{\Omega}_\mathcal{A}}
\newcommand{\OU}{\Omega_{\mathcal{X}_U/\mathbb{Z}[\frac{1}{n}]}}
\newcommand{\WA}{\underline{\omega}_\mathcal{A}}
\newcommand{\WU}{\omega_{\mathcal{X}_U/\mathbb{Z}[\frac{1}{n}]}}
\newcommand{\HHom}{\mathcal{H}\mathrm{om}}

\newcommand{\pair}[1]{\langle {#1} \rangle}
\newcommand{\wpair}[1]{\left\{{#1}\right\}}
\newcommand{\wh}{\widehat}
\newcommand{\wt}{\widetilde}

\newcommand\Spf{\mathrm{Spf}}

\newcommand{\lra}{{\longrightarrow}}

\newcommand{\matrixx}[4]
{\left( \begin{array}{cc}
  #1 &  #2  \\
  #3 &  #4  \\
 \end{array}\right)}        

\newcommand{\BA}{{\mathbb {A}}}
\newcommand{\BB}{{\mathbb {B}}}
\newcommand{\BC}{{\mathbb {C}}}
\newcommand{\BD}{{\mathbb {D}}}
\newcommand{\BE}{{\mathbb {E}}}
\newcommand{\BF}{{\mathbb {F}}}
\newcommand{\BG}{{\mathbb {G}}}
\newcommand{\BH}{{\mathbb {H}}}
\newcommand{\BI}{{\mathbb {I}}}
\newcommand{\BJ}{{\mathbb {J}}}
\newcommand{\BK}{{\mathbb {K}}}
\newcommand{\BL}{{\mathbb {L}}}
\newcommand{\BM}{{\mathbb {M}}}
\newcommand{\BN}{{\mathbb {N}}}
\newcommand{\BO}{{\mathbb {O}}}
\newcommand{\BP}{{\mathbb {P}}}
\newcommand{\BQ}{{\mathbb {Q}}}
\newcommand{\BR}{{\mathbb {R}}}
\newcommand{\BS}{{\mathbb {S}}}
\newcommand{\BT}{{\mathbb {T}}}
\newcommand{\BU}{{\mathbb {U}}}
\newcommand{\BV}{{\mathbb {V}}}
\newcommand{\BW}{{\mathbb {W}}}
\newcommand{\BX}{{\mathbb {X}}}
\newcommand{\BY}{{\mathbb {Y}}}
\newcommand{\BZ}{{\mathbb {Z}}}

\newcommand{\CA}{{\mathcal {A}}}
\newcommand{\CB}{{\mathcal {B}}}
\newcommand{\CD}{{\mathcal{D}}}
\newcommand{\CE}{{\mathcal {E}}}
\newcommand{\CF}{{\mathcal {F}}}
\newcommand{\CG}{{\mathcal {G}}}
\newcommand{\CH}{{\mathcal {H}}}
\newcommand{\CI}{{\mathcal {I}}}
\newcommand{\CJ}{{\mathcal {J}}}
\newcommand{\CK}{{\mathcal {K}}}
\newcommand{\CL}{{\mathcal {L}}}
\newcommand{\CM}{{\mathcal {M}}}
\newcommand{\CN}{{\mathcal {N}}}
\newcommand{\CO}{{\mathcal {O}}}
\newcommand{\CP}{{\mathcal {P}}}
\newcommand{\CQ}{{\mathcal {Q}}}
\newcommand{\CR }{{\mathcal {R}}}
\newcommand{\CS}{{\mathcal {S}}}
\newcommand{\CT}{{\mathcal {T}}}
\newcommand{\CU}{{\mathcal {U}}}
\newcommand{\CV}{{\mathcal {V}}}
\newcommand{\CW}{{\mathcal {W}}}
\newcommand{\CX}{{\mathcal {X}}}
\newcommand{\CY}{{\mathcal {Y}}}
\newcommand{\CZ}{{\mathcal {Z}}}

\newcommand{\ab}{{\mathrm{ab}}}
\newcommand{\Ad}{{\mathrm{Ad}}}
\newcommand{\an}{{\mathrm{an}}}
\newcommand{\Aut}{{\mathrm{Aut}}}

\newcommand{\Br}{{\mathrm{Br}}}
\newcommand{\bs}{\backslash}
\newcommand{\bbs}{\|\cdot\|}

\newcommand{\Ch}{{\mathrm{Ch}}}
\newcommand{\cod}{{\mathrm{cod}}}
\newcommand{\cont}{{\mathrm{cont}}}
\newcommand{\cl}{{\mathrm{cl}}}
\newcommand{\criso}{{\mathrm{criso}}}
\newcommand{\de}{{\mathrm{d}}}
\newcommand{\dR}{{\mathrm{dR}}}
\newcommand{\df}{\mathrm{det}^*}
\newcommand{\disc}{{\mathrm{disc}}}
\newcommand{\Div}{{\mathrm{Div}}}
\renewcommand{\div}{{\mathrm{div}}}
\newcommand{\Dh}{\widehat{\mathcal{D}}}
\newcommand{\Ei}{\mathrm{Ei}}
\newcommand{\Eis}{{\mathrm{Eis}}}
\newcommand{\End}{{\mathrm{End}}}
\newcommand{\Fil}{\mathrm{Fil}}
\newcommand{\Frob}{{\mathrm{Frob}}}

\newcommand{\Gal}{{\mathrm{Gal}}}
\newcommand{\GL}{{\mathrm{GL}}}
\newcommand{\GO}{{\mathrm{GO}}}
\newcommand{\GSO}{{\mathrm{GSO}}}
\newcommand{\GSp}{{\mathrm{GSp}}}
\newcommand{\GSpin}{{\mathrm{GSpin}}}
\newcommand{\GU}{{\mathrm{GU}}}
\newcommand{\BGU}{{\mathbb{GU}}}

\newcommand{\Has}{\mathrm{hasse}}
\newcommand{\Hom}{{\mathrm{Hom}}}
\newcommand{\Hol}{{\mathrm{Hol}}}
\newcommand{\HC}{{\mathrm{HC}}}
\newcommand{\id}{\mathrm{id}}
\newcommand{\Img}{{\mathrm{Im}}}
\newcommand{\Ind}{{\mathrm{Ind}}}
\newcommand{\ine}{\mathrm{ine}}
\newcommand{\inv}{{\mathrm{inv}}}
\newcommand{\Isom}{{\mathrm{Isom}}}

\newcommand{\Jac}{{\mathrm{Jac}}}
\newcommand{\JL}{{\mathrm{JL}}}

\newcommand{\Ker}{{\mathrm{Ker}}}
\newcommand{\KS}{{\mathrm{KS}}}

\newcommand{\Lie}{{\mathrm{Lie}}}
\renewcommand{\mod}{\mathrm{mod}}
\newcommand{\mm}{\mathfrak{m}}
\newcommand{\new}{{\mathrm{new}}}
\newcommand{\Nm}{\mathrm{Nm}}
\newcommand{\NS}{{\mathrm{NS}}}

\newcommand{\ord}{{\mathrm{ord}}}
\newcommand{\ol}{\overline}
\newcommand{\otf}{\otimes^*}
\newcommand{\rank}{{\mathrm{rank}}}

\newcommand{\PGL}{{\mathrm{PGL}}}
\newcommand{\PSL}{{\mathrm{PSL}}}
\newcommand{\Pic}{\mathrm{Pic}}
\newcommand{\Prep}{\mathrm{Prep}}
\newcommand{\Proj}{\mathrm{Proj}}
\renewcommand{\Pr}{\mathcal{P}r}
\newcommand{\Picc}{\mathcal{P}ic}

\newcommand{\ram}{\mathrm{ram}}
\renewcommand{\Re}{{\mathrm{Re}}}
\newcommand{\Res}{{\mathrm{Res}}}
\newcommand{\red}{{\mathrm{red}}}
\newcommand{\reg}{{\mathrm{reg}}}
\newcommand{\sm}{{\mathrm{sm}}}
\newcommand{\sing}{{\mathrm{sing}}}
\newcommand{\SL}{\mathrm{SL}}
\newcommand{\SLL}{\widetilde{\mathrm{SL}}}
\newcommand{\SO}{\mathrm{SO}}
\newcommand{\Sp}{\mathrm{Sp}}
\newcommand{\spl}{\mathrm{spl}}
\newcommand{\Sym}{{\mathrm{Sym}}}
\newcommand{\Spec}{\mathrm{Spec}}
\renewcommand{\ss}{\mathrm{ss}}
\newcommand{\tor}{{\mathrm{tor}}}
\newcommand{\tr}{{\mathrm{tr}}}

\newcommand{\ur}{{\mathrm{ur}}}
\newcommand{\U}{\mathrm{U}}
\newcommand{\UU}{\mathrm{U}(1,1)}
\newcommand{\vol}{{\mathrm{vol}}}

\newcommand{\ds}{\displaystyle}

\begin{abstract}
This is the second of a series of three papers, in which we prove a formula expressing the modular height of a unitary Shimura variety over a CM number field in terms of the logarithmic derivative of the Hecke L-function associated with the CM extension. The main idea of our proof is to compare the holomorphic projection of the derivative of a certain mixed Eisenstein-theta series and the arithmetic degree of a generating series of divisors on unitary Shimura varieties.

In this paper, we define the arithmetic generating series of divisors on unitary Shimura varieties, compute the corresponding arithmetic intersection numbers, and derive the modular height formula for unitary Shimura curves as well as the height formula for a CM point on them.
\end{abstract}

\tableofcontents

\section{Introduction}
The goal of this series of three papers (\cite{Guo1}, \cite{Guo2} and the current one) is to prove a formula expressing the modular height of a unitary Shimura variety over a CM field in terms of the logarithmic derivative of the Hecke L-function associated with the CM extension. Our work can be viewed as an extension of X. Yuan's work \cite{Yuan1}, which is based on  the work of Yuan–Zhang–Zhang \cite{YZZ2} on the Gross–Zagier formula, and  the work of Yuan–Zhang \cite{YZ1} on the averaged Colmez conjecture. All these works are in turn inspired by  the pioneering work of Gross–Zagier \cite{GZ} and some philosophies of Kudla’s program \cite{Kud1,Kud2,Kud3,Kud4}.  These works aim to calculate the arithmetic invariants of Shimura varieties using special values of L-functions.

In our work, we will study the generating series of divisors on unitary Shimura varieties and their arithmetic versions, comparing them with the derivative of mixed theta-Eisenstein series. The comparison gives the modular-height formula. For a complete introduction to our work, we refer to the introduction to the third paper \cite{Guo2} in this series.

This is the second paper in this series. Its goal is to provide ingredients on the ``arithmetic side" needed for our proof. More precisely, we begin with the definition of unitary Shimura varieties and their integral models, and then define the generating series of divisors on them together with their arithmetic counterparts. By taking suitable arithmetic intersection numbers, we obtain an automorphic form, which we call the height series. We will complete a portion of the computations in the height series and compare them with the conclusions in \cite{Guo1}, thereby obtaining a partial arithmetic Siegel--Weil formula.

At the same time, we will also use the exceptional correspondence between Shimura curves to derive the modular height formula for unitary Shimura curves, as well as the height formula for a CM point on them. These two formulas will serve as the inductive foundation for our proof of the modular height formula in general dimensions.

We use the notation and analytic normalizations of \cite[Sections~2.1 and~4.1]{Guo1}. The integral Hodge line bundle and arithmetic intersection conventions are specified in Sections~\ref{Hodge bundle and CM cycle} and~\ref{Arakelov} below.

\subsection{Modular heights of CM points}

Throughout this series of papers, we always fix a totally real field $F\ne\QQ$ of degree $d=[F:\QQ]$ with a fixed infinite place $\iota$. Let $E/F$ be a fixed CM extension, i.e., a totally imaginary quadratic extension of $F$. Under a fixed embedding from $\RR$ to $\CC$, we can also view $\iota$ as a place of $E$. Let $\BV$ be a totally positive definite \textit{incoherent Hermitian space} over $\BA_E$ of dimension $n+1$ with $n>0$, i.e., there does not exist any Hermitian space $V$ over $E$ such that $V\otimes_E \BA_E=\BV$. We also define $V$ to be the \textit{nearby coherent Hermitian space} with respect to $\iota$, i.e., $V$ has signature $(n,1)$ at $\iota$ and $(n+1,0)$ at all other archimedean places.

Let $G=\Res_{F/\QQ}\U(V)$ be a reductive group over $\QQ$, where $\U(V)$ is the unitary group of the Hermitian space $V$.  For this $\QQ$-group we write $G(\Af)=G(\BA_{\QQ,f})=\U(V)(\BA_{F,f})$. The Hermitian symmetric domain $D$ is defined as follows:
\begin{equation*}
    D=\{z\in\BP(V_{\iota,\CC})\big|q(z)<0\}.
\end{equation*}
Here $q(z)<0$ means that any nonzero vector spanning the line $z$ has negative Hermitian norm. In particular, $D$ is connected and carries a $\U(V_\CC)$-invariant complex structure.

Choose a certain Hermitian lattice $\Lambda\subset V$, which is self-dual at every finite place unramified in $E/F$, and is $\varpi_{E_v}$-modular or almost $\varpi_{E_v}$-modular at every ramified place; see \cite[Section~2.1, Definition~2.2]{Guo1}. Let $U$ be the open compact subgroup of $G(\Af)$ given by the stabilizer of $\Lambda$ at every finite place. Then we define the associated \textit{unitary Shimura variety} $X_U$ over $E$, whose complex points are given by

\begin{equation*}
    X_U(\CC)=G(\QQ)\backslash D\times G(\wh{\QQ})/U.
\end{equation*}

We regard $X_U$ as a smooth Deligne--Mumford stack over $E$ of dimension $n$. Line bundles and degrees on its coarse space carry the corresponding rational stabilizer weights. Local intersection computations are made on representative \'etale charts and descend with these weights.

Under the above complex uniformization, let $L_D$ be the tautological line bundle on $D$, i.e., for each $z\in D$, the fiber $L_{D,z}$ is isomorphic to $\CC z$. This line bundle carries a natural Hermitian metric $h_{L_D}$ such that
\begin{equation*}
    h_{L_D}(s_z)=-q(s_z),
\end{equation*}
for $s_z\in L_{D,z}\cong\CC z$, $z\in D$. Thus $h_{L_D}(s_z)=\|s_z\|^2$. The metric is equivariant under the action of $\U(V_\CC)$. Let $L_U$ be the descent of the tautological line; it is an ample line bundle on $X_U$ with $\QQ$-coefficients. We call $L_U$ the \textit{Hodge bundle} of $X_U$.

Under Assumption~\ref{Assumption}, the RSZ construction gives the regular proper integral model used here; see \cite[Theorem~4.4.4 and Corollary~4.4.8]{Qiu}. We denote its metrized Hodge line bundle by $\LL_U$. Its integral definition, including the almost modular extension, is fixed in Section~\ref{Hodge bundle and CM cycle}. The standard lattice datum and its compatible CM datum are standing choices throughout the higher-dimensional calculation.

In order to define a CM point, consider an orthogonal decomposition
\begin{equation*}
    \BV=\BW\oplus W^\perp(\BA_E),
\end{equation*}
where $\BW$ is an incoherent totally positive definite Hermitian subspace of dimension 1,  such that their Hermitian determinant classes agree in $F_v^\times/\Nm(E_v^\times)$ at every place $v$. Then the orthogonal complement $W^\perp$ is coherent.

Denote $U_\BW=U\cap\U(\BW_{\Af})$. The zero-dimensional Shimura variety attached to $\BW$ maps finitely to $X_U$. Its pushforward, divided by its degree and with the source stabilizer weights, is the degree-one CM cycle $P_{\BW,U}$. Its arithmetic closure is denoted by $\mathcal P_{\BW,U}$, or by $\mathcal P$. On a chosen complex lift the negative line is represented by $[(e),1]$; the complete cycle includes all CM double classes. The component weights are given in \eqref{CM component weights}.

In this paper, we will prove the following formula.

\begin{theorem}\label{Modular height of CM point}
Let $\eta$ be the quadratic character associated with $E/F$, and let $L_f(s,\eta)$ be the finite Hecke $L$-function. For the standard lattice and compatible CM datum of Assumption~\ref{Assumption},
\begin{equation*}
    \LL\cdot\mathcal P=-2\frac{L_f'(0,\eta)}{L_f(0,\eta)}.
\end{equation*}
\end{theorem}

CM heights enter the averaged Colmez formula through their comparison with Faltings heights on quaternionic Shimura curves in \cite{YZ1}. The independent proof of \cite{AGHMP} uses CM cycles on orthogonal Shimura varieties. Related comparisons for the Colmez conjecture are studied in \cite{Zh2}.

The CM height also occurs in the arithmetic generating series of \cite{Qiu}. Theorem~\ref{Modular height of CM point} computes it for the line bundle $\LL$. The constant $c_3$ in the projected derivative of the first paper is computed in \cite[Theorem~3.4 and Section~4]{Guo1}.

\subsection{Modular heights of unitary Shimura curves}\label{curve introduction}

Keep all the notations as above. We now additionally assume that $n=1$, in which case the Shimura variety is a unitary Shimura curve. At the same time, for this curve calculation we relax the earlier requirement on the Hermitian lattice and allow it to be almost self-dual at finitely many places that are inert in $E/F$. In this rank-two case, such a lattice has an orthogonal basis $e_1,e_2$ with
\begin{equation*}
    v(q(e_1))=0,\qquad v(q(e_2))=1.
\end{equation*}
We denote by $\Sigma_f$ the set consisting of these places, which is a finite set.

The \textit{modular height} of $X_U$ with respect to $\LL_U$ is defined to be
\begin{equation*}
    h_{\LL_U}(X_U)=\frac{\wh{\deg}(\LL_U^2)}{\deg(L_U)}.
\end{equation*}

Here $\deg(L_U)$ is the degree of $L_U$ on the generic curve, and the numerator is its arithmetic self-intersection number on $\mathcal X_U$.

In this paper, we will prove the following  formula for the modular height of unitary Shimura curves.

\begin{theorem}\label{Modular height of curve}
Let $n=1$, and allow inert almost self-dual lattices at the places in $\Sigma_f$. Put $d_\BB=\prod_{v\in\Sigma_f}N_v$. With the integral Hodge line bundle of Section~\ref{Hodge bundle and CM cycle},
\begin{equation*}
    \begin{aligned}
        h_{\LL}(X) ={}&-2\frac{L_f'(0,\eta)}{L_f(0,\eta)} +2\frac{\zeta_F'(2)}{\zeta_F(2)}+2\log|d_F|\\
        &-\bigl(2\gamma+2\log(2\pi)-1\bigr)[F:\QQ] +\log d_{E/F}+\log d_\BB.
    \end{aligned}
\end{equation*}
Here $\zeta_F$ is the Dedekind zeta function, $d_F$ is the absolute discriminant of $F$, and $d_{E/F}$ is the norm of the relative discriminant. The associated degree-one CM cycle on this curve satisfies
\begin{equation*}
    \LL\cdot\mathcal P=-2\frac{L_f'(0,\eta)}{L_f(0,\eta)}+\log d_\BB.
\end{equation*}
\end{theorem}

The case $\Sigma_f=\varnothing$ gives the initial curve case for \cite{Guo2}. For the additional lattices at $v\in\Sigma_f$, Proposition~\ref{comparison of Hodge bundles on curves} computes the degrees of $\mathcal L$ on the components of the special fibre.

\subsection{Main theorem of generating series of divisors}

In our proof of the modular height formula, the generating series formed by special divisors on the unitary Shimura variety and its arithmetic version are very important. They form the core component of the height series on the arithmetic side. From a more general perspective, generating series are also important tools for studying special cycles on Shimura varieties. Here we give a brief introduction before stating the comparison used in this paper.

Keep the notation above. Let $\Phi=\Phi_\infty\otimes\Phi_f$, with the fixed archimedean Gaussian and $\Phi_f\in\mathcal S(\BV_f)^U$. The finite weights define the divisor series on $X_U$:

\begin{equation*}
    Z_\Phi(\tau)=[L^\vee]\Phi_f(0)+\sum_{t\in F_+}Z_t q^t.
\end{equation*}
Here $q=(e^{2\pi i \tau_1},\cdots,e^{2\pi i \tau_d})$ with $d=[F:\QQ]$ and $\tau=(\tau_k)^d_{k=1}\in\mathcal{H}^d$, and $Z_t$ is the weighted special divisor
\begin{equation*}
    Z_t=\sum_{x\in U\backslash V_f,\langle x,x\rangle=t}\Phi_f(x)Z(x)_U
\end{equation*}
with $Z(x)_U$ the Kudla special divisor on $X_U$ associated with $x$. This is a Hilbert modular form of weight $\mathfrak{m}$. Here $\mathfrak{m}=(\mathfrak{m}_v)_{v|\infty}$, where $\mathfrak{m}_v$ is a pair of integers that is defined in \cite[Section~3.2]{Guo1}. Note that we can also write such generating function in terms of
\begin{equation*}
    Z(g,\Phi)=r(g)\Phi(0)[L^\vee]+\sum_{t\in F^+}\sum_{y\in U\backslash V_f,\langle y,y\rangle=t}r(g)\Phi(y)Z(y)_U
\end{equation*}

Here $g\in\UU(\BA_F)$ and $r(g)$ is the Weil representation. It should be noted that each special divisor can be viewed as the pushforward of a unitary Shimura variety of codimension one.

Geometric modularity for special-cycle generating series is established in several settings, including the Chow-valued constructions of \cite{Bor,KRY,YZZ1,Liu1}. We use the unitary divisor series of \cite{Liu1}.

Applying intersection functionals to these Chow-valued forms produces scalar automorphic forms. In particular, the degree functional gives the geometric Siegel--Weil formula used in the present comparison. According to \cite{Kud1}, we have
\begin{equation*}
    \deg_L(Z(g,\Phi)):=Z(g,\Phi)\cdot c_1(L)^{n-1}=-\deg_L(X_U)E(0,g,\Phi),
\end{equation*}
Here $\cdot$ denotes intersection. We recall this formula and its references in \eqref{Geometric Siegel Weil}.

For the arithmetic comparison, we consider a generating series of arithmetic divisor classes on the integral model:
\begin{equation*}
    \wh{\mathcal{Z}}_\Phi(\tau):=[\LL^\vee]\Phi_f(0)+\sum_{t\in F_+}\wh{\mathcal{Z}}_t q^t
\end{equation*}
Kudla's modularity problem asks for modularity of such a series with suitable arithmetic extensions and constant term; see \cite{Kud3}.

Arithmetic modularity has been proved for a range of divisor and one-cycle generating series, including quaternionic Shimura curves \cite{KRY}, Hilbert modular surfaces \cite{BBK}, orthogonal Shimura varieties \cite{HP}, and unitary Shimura varieties over imaginary quadratic fields \cite{BHK+}.

Following the admissible-extension construction of \cite[section ``Divisors and 0-cycles'']{Zh1} and \cite{Qiu}, we choose normalized $\LL$-admissible extensions of the generic divisor classes. Section~\ref{Arakelov} specifies their curvature and mean-zero conditions. At finite places, the normalized representative is identified with its weighted Kudla--Rapoport branches before their local intersections are computed. Pairing with a componentwise degree-zero cycle then gives the automorphic height series used below.

The main comparison of this paper pairs the admissible generating series with the complete CM cycle and compares its positive Fourier coefficients with the projected derivative of the mixed theta--Eisenstein series.

\begin{theorem}\label{main theorem of arithmetic Siegel-Weil I}
Assume the standard datum of Assumption~\ref{Assumption} and the compatible Schwartz function of \cite[Section~4.1]{Guo1}. Let $g=(g_\infty,1)$, and let $\wh{\mathcal Z}_*(g,\Phi)$ be the positive-index part of the normalized admissible generating series. Put $H_m=\sum_{j=1}^m1/j$, with $H_0=0$. There is a finite sum $A(g,\Phi)$ of non-singular degenerate pseudo-theta series supported on $W^\perp$ such that, coefficientwise for positive Fourier indices,
\begin{equation*}
    \begin{aligned}
        &\Pr'I'(0,g,\Phi)+\wh{\mathcal Z}_*(g,\Phi)\cdot\mathcal P\\
        ={}&-[F:\QQ]\bigl(\gamma+\log(4\pi)-2H_{n-1}\bigr)E_*(0,g,\Phi)\\
        &-\sum_{\substack{v\nmid\infty\\v\ \mathrm{nonsplit}}} \avint_{E^1\backslash E^1(\BA)} \theta_*(g,\tau,\bar k_{\Phi_v}\otimes\Phi^v)\,d\tau+A(g,\Phi).
    \end{aligned}
\end{equation*}
Here $\Pr'I'$ is the projection of \cite[Theorem~3.3]{Guo1}, $\bar k_{\Phi_v}$ is the Schwartz extension of $2k_{\Phi_v}(1,y)-m_{\Phi_v}(y)\log N_v$ constructed in Section~\ref{The i-part at non-archimedean places}, and $\theta_*$ denotes the positive Fourier part of its theta series.
\end{theorem}

Here the symbol $\avint$ denotes division of the integral by $2L(1,\eta)$, the volume of $E^1\backslash E^1(\BA)$. The sum over finite places has finite support, and the identity remains valid when the finite identity components are replaced by elements of sufficiently small compact open stabilizers of these local data.

Section~\ref{Proof of the main theorem} combines the containing contribution with the boundaries of the proper kernels and proves the nonsingularity of $A(g,\Phi)$. In \cite{Guo2}, the arithmetic degree and the remaining projected derivative complete the automorphic difference, to which the key lemma of \cite[Section~2.3]{Guo1} applies.

This is the part of the so-called ``kernel identity" obtained by pairing with the CM cycle in the language of \cite{YZZ2}. The analytic decomposition in \cite[Section~3.2]{Guo1} separates the projected derivative into $\Pr'I'(0,g,\Phi)$ and $\Pr'\mathcal J'(0,g,\Phi)$. The former is compared here with $\wh{\mathcal Z}_*\cdot\mathcal P$; the latter is combined in \cite{Guo2} with the arithmetic degree of the generating series. Combining these two arithmetic pairings with the geometric degree correction gives the height series of Section~\ref{Height series}.

Related arithmetic Siegel--Weil formulas compare arithmetic intersection numbers with derivatives of Eisenstein series. Li--Zhang \cite[Theorem~1.3.1]{LZ} prove such a formula for full-rank Kudla--Rapoport cycles, and Chen \cite{Chen} studies the corank-one case.

In a more general perspective, the results of Li--Liu \cite{LL1,LL2} relate nonzero derivatives of $L$-functions to nontrivial Chow classes through arithmetic theta lifting. They support the Beilinson--Bloch conjecture.

\subsection{Arrangement of this paper}

The arrangement of this paper is as follows. In $\S$~\ref{Shimura}, we give the precise definition of unitary Shimura varieties together with their integral models, as well as the CM cycles and Hodge bundles on them. In $\S$~\ref{Unitary Shimura curve}, we focus on Shimura curves and prove Theorems~\ref{Modular height of CM point} and~\ref{Modular height of curve} by comparing unitary and quaternionic Shimura curves.

In $\S$~\ref{Height series}, we give the precise definition of the generating series of special divisors and its admissible arithmetic extension, thereby defining the height series. In $\S$~\ref{Explicit local intersection}, we compute the arithmetic intersection numbers locally at each place and combine them with the main theorem of \cite{Guo1} to prove Theorem~\ref{main theorem of arithmetic Siegel-Weil I}.

\subsection*{Acknowledgement}

The author is deeply grateful for the valuable assistance and meticulous guidance provided by Professor Xinyi Yuan. Indeed, it is thanks to his previous work with Shou-Wu Zhang and Wei Zhang that the author has had the privilege to build upon it and make further contributions. He would like to thank his friend Weixiao Lu for many helpful advice. He thanks Ryan Chen, Yu Luo, Yinchong Song, Liang Xiao and Roy Zhao for helpful communication. He also thanks Yifeng Liu and Wei Zhang for some useful suggestions. Finally, the author is grateful to the anonymous referee for so many valuable comments or suggestions to revise this paper.

\section{Unitary Shimura varieties and integral models}\label{Shimura}
In this section we recall the unitary, unitary-similitude, and RSZ Shimura varieties and their reflex fields.  The arithmetic constructions are carried out on the RSZ model.  The signature is the strict fake Drinfeld type of \cite{RSZ2}.

In the following discussion, recall that we always fix the totally real field $F\ne\QQ$, the CM extension $E/F$, and the infinite place $\iota$ as above. For convenience, we also fix a CM type $\Phi$ of $E$, i.e., $\Phi$ consists of $d$ distinct archimedean places of $E$ such that $\Phi\cap\overline{\Phi}=\emptyset$. By abuse of notation, we also say that $\iota\in\Phi$.

Let $\BV$ be the totally positive definite incoherent Hermitian space over $\BA_E$ of rank $n+1$, and let $V$ be its nearby coherent space at $\iota$, as in \cite[Section~2.1]{Guo1}. The associated Shimura variety has dimension $n$.

\subsection{Unitary Shimura varieties}

In this subsection, we provide the specific definitions of these various unitary Shimura varieties and explain the maps between them.

\subsubsection*{Unitary Shimura variety in our setting}

We first give a precise definition of the unitary Shimura variety we consider. Throughout this paper, the words \textit{unitary Shimura variety} always refer to this Shimura datum.

Let $G=\Res_{F/\QQ}\U(V)$ which is a reductive group over $\QQ$. The Hermitian symmetric domain $D$ is defined as follows:
\begin{equation*}
    D=\{z\in\BP(V_{\iota,\CC})\big|q(z)<0\}.
\end{equation*}
The domain $D$ is connected and its complex structure is $G(\RR)$-invariant.

Let $U$ be an open compact subgroup of $G(\wh{\QQ})$. Then we can define a Shimura variety $X_U$ with $\CC$-points
\begin{equation*}
    X_U(\CC)=G(\QQ)\backslash D\times G(\wh{\QQ})/U.
\end{equation*}
This Shimura stack is smooth over the reflex field $E$ of dimension $n$. Alternatively, the conjugacy class in the Shimura datum for $G$ is the conjugacy class of the homomorphism $h_G=(h_{G,\varphi})_{\varphi\in\Phi}$, where
\begin{equation*}
    h_{G,\varphi}: z\mapsto \mathrm{diag}(1,\cdots,1),\quad (\varphi\ne\iota)
\end{equation*}
and
\begin{equation*}
    h_{G,\iota}: z\mapsto \mathrm{diag}(1,\cdots,1,\overline{z}/z)
\end{equation*}
under the inclusion
\begin{equation*}
    G(\RR)\subset\GL_{E\otimes\RR}(V\otimes\RR)\stackrel{\sim}{\longrightarrow}\prod_{\varphi\in\Phi}\GL_\CC(V_\varphi).
\end{equation*}
Thus, we can also define a Shimura variety $S(G,\{h_G\})_U=X_U$ in this way.

The Hodge line bundle $L_U$ is the descent of the tautological line $L_D$ on $D$, with metric

\begin{equation*}
    h_{L_D}(s_z)=-\langle s_z,s_z\rangle,
\end{equation*}
for $s_z\in L_{D,z}$. Its descent is an ample $\QQ$-line bundle on $X_U$. The metric in \cite[(5.2.6)]{BH} is $(4\pi e^\gamma)^{-1}$ times this metric.

This datum is of abelian type. We use the related PEL data below to construct its integral model.
\subsubsection*{Shimura variety of unitary similitudes}

Let $\GU(V)$ be the unitary similitude group over $F$, so that for every $F$-algebra $R$,
\begin{equation*}
    \GU(V)(R)=\left\{g\in\GL_{E\otimes_F R}(V\otimes_F R)\ \middle|\ \langle gx,gy\rangle=c(g)\langle x,y\rangle\text{ for some }c(g)\in R^\times\right\}.
\end{equation*}
The group occurring in the PEL Shimura datum is the following $\QQ$-group:
\begin{equation*}
    G^\QQ:=\BG_m\times_{\Res_{F/\QQ}\BG_m}\Res_{F/\QQ}\GU(V).
\end{equation*}
Here $\BG_m\to\Res_{F/\QQ}\BG_m$ is the diagonal map and the other map is induced by $c$. Equivalently,
\begin{equation*}
    G^\QQ(\QQ)=\{g\in\GU(V)(F)\mid c(g)\in\QQ^\times\}.
\end{equation*}

The signatures of $V$ determine a generalized CM type $r$ of rank $n+1$ on $E$.  With respect to the fixed CM type $\Phi$ and the distinguished embedding $\iota\in\Phi$, it is the strict fake Drinfeld type
\begin{equation*}
    r_\iota=n,\qquad r_\varphi=n+1\quad(\varphi\in\Phi\setminus\{\iota\}), \qquad r_{\bar\varphi}=n+1-r_\varphi.
\end{equation*}
The Shimura homomorphism has components
\begin{equation*}
    h_{G^\QQ,\varphi}(z)=\operatorname{diag} \big(z\,1_{r_\varphi},\bar z\,1_{n+1-r_\varphi}\big), \qquad \varphi\in\Phi.
\end{equation*}
Its reflex field is $E_r$, the reflex field of the generalized CM type $r$.

Let $K\subset G^\QQ(\Af)$ be compact open.  The Kottwitz moduli problem is a category over $(\mathrm{LNSch})_{/E_r}$.  It parametrizes quadruples $(A,\psi,\lambda,\bar\eta)$ where $A$ is an abelian scheme of relative dimension $(n+1)[F:\QQ]$, $\psi:E\to\End^0(A)$ is an action, $\lambda$ is a quasi-polarization, and $\bar\eta$ is a $K$-orbit of $\BA_{E,f}$-linear symplectic similitudes.  We impose
\begin{equation*}
    \operatorname{Ros}_\lambda(\psi(a))=\psi(\bar a),\qquad a\in E,
\end{equation*}
and the Kottwitz determinant condition
\begin{equation*}
    \operatorname{char}\big(\psi(a)\mid\Lie A;T\big) =\prod_{\varphi\in\Hom_\QQ(E,\overline\QQ)}(T-\varphi(a))^{r_\varphi}.
\end{equation*}
Morphisms are $E$-linear quasi-isogenies preserving the level structure and carrying the polarization to a locally constant $\QQ^\times$-multiple.

\begin{theorem}
The above moduli problem is represented by a Deligne--Mumford stack $M_K$ over $\Spec E_r$.  If $n$ is odd, its complex fiber is identified with
\begin{equation*}
    S(G^\QQ,\{h_{G^\QQ}\})_K.
\end{equation*}
If $n$ is even, it is a finite disjoint union of copies of this Shimura variety, indexed by
\begin{equation*}
    \ker^1(\QQ,G^\QQ)=\ker\!\left(H^1(\QQ,G^\QQ)\longrightarrow H^1(\BA,G^\QQ)\right).
\end{equation*}
The tower $(M_K)_K$ is the canonical model; see \cite[Theorem~2.2]{RSZ2}.
\end{theorem}

\subsubsection*{The RSZ Shimura varieties}

Let $V_0$ be a one-dimensional totally positive definite $E/F$-Hermitian space. Set
\begin{equation*}
    Z^\QQ=\{z\in\Res_{E/\QQ}\BG_m\mid\Nm_{E/F}(z)\in\BG_m\}.
\end{equation*}
The toric Shimura datum $(Z^\QQ,\{h_{Z^\QQ}\})$ has reflex field $E_\Phi$, the reflex field of the CM type $\Phi$. Define
\begin{equation*}
    \wt G:=Z^\QQ\times_{\BG_m}G^\QQ,
\end{equation*}
where the maps to $\BG_m$ are $\Nm_{E/F}$ and the similitude character.  The reflex field of the RSZ datum is
\begin{equation*}
    \wt E=E_\Phi E_r=E_\Phi E_{r^\natural}.
\end{equation*}
Here $E_{r^\natural}$ is the reflex field of the unitary datum $(G,\{h_G\})$.  In the strict fake Drinfeld case, the distinguished embedding $\iota$ identifies $E_{r^\natural}$ with the CM field $E$. Thus $X_U$ is defined over $E$, whereas the similitude and RSZ varieties are generally defined over the larger fields $E_r$ and $\wt E$.

\begin{proposition}

The projection and product decomposition give the following relations between the three Shimura data.

The projection $\wt G\to G^\QQ$ induces, after base change, a morphism
\begin{equation*}
    S(\wt G,\{h_{\wt G}\})_{K_{\wt G}} \longrightarrow S(G^\QQ,\{h_{G^\QQ}\})_{K_{G^\QQ}}\otimes_{E_r}\wt E.
\end{equation*}
The central embedding $Z^\QQ\hookrightarrow G^\QQ$ gives an isomorphism
\begin{equation*}
    \wt G\stackrel\sim\longrightarrow Z^\QQ\times G, \qquad (z,g)\longmapsto(z,z^{-1}g),
\end{equation*}
and hence, for product level $K_{\wt G}=K_{Z^\QQ}\times U$, an isomorphism of canonical models over $\wt E$,
\begin{equation}\label{unitary and RSZ}
    S(\wt G,\{h_{\wt G}\})_{K_{\wt G}} \cong \big(S(Z^\QQ,\{h_{Z^\QQ}\})_{K_{Z^\QQ}}\otimes_{E_\Phi}\wt E\big) \times_{\wt E}\big(X_U\otimes_E\wt E\big).
\end{equation}

Let $\wt D=D(V_0)\times D(V)\cong D$.  The complex uniformization is
\begin{equation*}
    S(\wt G,\{h_{\wt G}\})_{K_{\wt G}}(\CC) =\wt G(\QQ)\backslash\wt D\times\wt G(\Af)/K_{\wt G},
\end{equation*}
and projection to the second factor gives
\begin{equation*}
    \wt G(\QQ)\backslash\wt D\times\wt G(\Af)/K_{\wt G} \longrightarrow G(\QQ)\backslash D\times G(\Af)/U.
\end{equation*}
\end{proposition}

Write $\wt V=\Hom_E(V_0,V)$.  The RSZ moduli problem over $(\mathrm{LNSch})_{/\wt E}$ parametrizes tuples
\begin{equation*}
    (A_0,\psi_0,\lambda_0,\bar\eta_0;A,\psi,\lambda,\bar\eta),
\end{equation*}
where the first quadruple lies in a fixed component of the toric moduli stack (base changed from $E_\Phi$ to $\wt E$), the second abelian scheme has signature type $r$, and the level structure identifies
\begin{equation*}
    \Hom_{\BA_{E,f}}(\widehat V(A_0),\widehat V(A)) \stackrel\sim\longrightarrow\wt V\otimes_E\BA_{E,f}
\end{equation*}
as Hermitian modules.  The Hermitian form on the left is
\begin{equation*}
    h(x,y)=\lambda_0^{-1}\circ y^\vee\circ\lambda\circ x.
\end{equation*}

\begin{theorem}
The RSZ moduli problem is represented by a Deligne--Mumford stack $M_{K_{\wt G}}(\wt G)$ over $\Spec\wt E$, whose complex fiber is the canonical model $S(\wt G,\{h_{\wt G}\})_{K_{\wt G}}$; see \cite[Theorem~3.5]{RSZ2}.
\end{theorem}
Note that when $V_0=E$ with its norm form, $\wt V\cong V$.

\subsubsection*{Variant moduli problems of the RSZ Shimura variety}

Assume that
\begin{equation*}
    K_{Z^\QQ}=\{z\in(\mathcal O_E\otimes\widehat\ZZ)^\times\mid \Nm_{E/F}(z)\in\widehat\ZZ^\times\}.
\end{equation*}
For a nonzero ideal $\mathfrak a\subset\mathcal O_F$, let $M_0^{\mathfrak a}$ be the toric moduli stack over $E_\Phi$ parametrizing triples $(A_0,\psi_0,\lambda_0)$, where $A_0$ has relative dimension $[F:\QQ]$, the $\mathcal O_E$-action has CM type $\Phi$, and $\ker\lambda_0=A_0[\mathfrak a]$. It is finite and \'etale over $E_\Phi$ and decomposes as
\begin{equation*}
    M_0^{\mathfrak a}=\coprod_\xi M_0^{\mathfrak a,\xi}.
\end{equation*}
Choose $\mathfrak a$ and $\xi$ so that $M_0^{\mathfrak a,\xi}$ realizes the fixed toric component for $K_{Z^\QQ}$, and base change it to $\wt E$. Thus the variant functor over $(\mathrm{LNSch})_{/\wt E}$ parametrizes
\begin{equation*}
    (A_0,\psi_0,\lambda_0,A,\psi,\lambda,\bar\eta),
\end{equation*}
with $(A_0,\psi_0,\lambda_0)\in M_0^{\mathfrak a,\xi}$ and $(A,\psi,\lambda,\bar\eta)$ as above. Recovering the toric level structure gives a natural equivalence
\begin{equation*}
    \mathcal F'_{K_{\wt G}}(\wt G)\stackrel\sim\longrightarrow \mathcal F_{K_{\wt G}}(\wt G).
\end{equation*}
We form the maps between these stacks after base change to the common reflex field $\wt E$.

\subsection{RSZ integral models and the normalization convention}\label{Integral model}

We use the PEL-type RSZ model and its relative local models for the arithmetic constructions.

\begin{assumption}\label{Assumption}
\begin{enumerate}
\item The extension $E/\QQ$ is Galois, or $E$ is the compositum of $F$ with an imaginary quadratic field.
\item Every finite place $v\mid 2$ of $F$ is split in $E/F$.  Every finite place of $F$ ramified over $\QQ$ is unramified in $E/F$.
\item If $v$ is inert in $E/F$, then $\BV_v$ contains a self-dual lattice $\Lambda_v$.
\item If $v$ is ramified in $E/F$, then $\Lambda_v$ is $\varpi_{E_v}$-modular for odd $n$ and almost $\varpi_{E_v}$-modular for even $n$.
\end{enumerate}
\end{assumption}

These hypotheses define the standard datum used for the higher-dimensional height and local intersection formulas. The relaxed curve datum of Section~\ref{curve introduction} is treated separately at its inert almost self-dual places.

Let $K_{Z^\QQ}$ stabilize a rank-one Hermitian lattice $\Lambda_0\subset V_0$, let $U$ stabilize $\Lambda$, and put

\begin{equation*}
    K_{\wt G}=K_{Z^\QQ}\times U, \qquad \wt\Lambda=\Hom_{\mathcal O_E}(\Lambda_0,\Lambda)\subset\wt V.
\end{equation*}

At a prime $p$, let $\mathfrak A\subset\mathfrak B$ be the local embedding sets attached to the strict fake Drinfeld signature, so that
\begin{equation*}
    \mathfrak B=\mathfrak A\sqcup\{\varphi_0,\bar\varphi_0\}.
\end{equation*}
We denote the corresponding Eisenstein ideals by $J_{\mathfrak A}$ and $J_{\mathfrak B}$.

For an abelian scheme $A/S$, let $\mathbb D(A)$ be the vector bundle dual to $H^1_{\mathrm{dR}}(A/S)$ and write
\begin{equation*}
    0\longrightarrow\Fil(A)\longrightarrow\mathbb D(A) \longrightarrow\Lie(A)\longrightarrow0.
\end{equation*}
At a place $\wt\nu\mid p$ of $\wt E$, the completed RSZ functor parametrizes tuples

\begin{equation*}
    (A_0,\psi_0,\lambda_0,A,\psi,\lambda,\text{level data})
\end{equation*}
with the usual polarization and level conditions, together with
\begin{equation*}
    \Fil(A_0)=J_\Phi\mathbb D(A_0), \qquad J_{\mathfrak B}\mathbb D(A)\subset\Fil(A)\subset J_{\mathfrak A}\mathbb D(A),
\end{equation*}
and the signature $(n,1)$ condition on $J_{\mathfrak A}\Lie(A)$.

\begin{theorem}
Under Assumption~\ref{Assumption}, the RSZ moduli problem is represented by a Deligne--Mumford stack over the corresponding integral reflex ring.  At a place $v$ ramified in $E/F$, its completed local model is the RSZ model attached to the chosen $\varpi_{E_v}$-modular or almost $\varpi_{E_v}$-modular lattice chain.  If $E_v/F_v$ is unramified and $F_v/\QQ_p$ is ramified, the comparison theorem of \cite{LMZ} identifies the completion with the corresponding relative strict local model.
\end{theorem}

\begin{proof}
The ramified local models are those of \cite{RSZ1,RSZ2}; the relative comparison at places ramified over $\QQ$ is given by \cite[Sections~3--6]{LMZ}.
\end{proof}

\begin{lemma}
For the standard maximal-lattice datum, the integral model is smooth at every finite place. The normalized finite admissible extension of a generic divisor is its Zariski closure.
\end{lemma}

\begin{proof}
The smoothness statement is \cite[Lemma~4.4.6]{Qiu}. At an unramified relative factor it is preserved by the comparison in \cite[Theorem~5.12 and Corollary~5.14]{LMZ}; the ramified factors are the smooth models of \cite[Sections~6--7]{RSZ1}. The assertion about the normalized extension follows from \cite[Lemma~A.2.4]{Qiu}.
\end{proof}

\begin{remark}
The local moduli problem is relative over $\mathcal O_{F_v}$, and the duality in the relative category is Faltings duality.  We use the corresponding relative Dieudonn\'e or display description whenever $F_v/\QQ_p$ is ramified.
\end{remark}

\begin{remark}
At a place ramified in $E/F$, the residue characteristic is odd and $F_v/\QQ_p$ is unramified. We use the modular model of \cite[Section~6]{RSZ1} when $n$ is odd and the almost modular model of \cite[Section~7]{RSZ1} when $n$ is even. The Hodge line bundle is defined in \eqref{ramified line bundle of modular forms} and \eqref{almost modular Hodge extension}.
\end{remark}

\begin{remark}

At the maximal levels of Assumption~\ref{Assumption} and at the curve levels of Section~\ref{curve introduction}, we divide the arithmetic degrees by the toric and reflex-field degrees specified below. At smaller levels the Hodge line bundle is pulled back from the chosen maximal-level model.

\end{remark}

We now fix the normalization. Let $T^{\xi}_{K_{Z^\QQ}}$ be the chosen zero-dimensional toric component and put
\begin{equation*}
    \delta_Z=\deg(T^{\xi}_{K_{Z^\QQ}}/E_\Phi),
\end{equation*}
with the degree weighted by the stabilizers. By \eqref{unitary and RSZ}, the generic fiber has a projection

\begin{equation*}
    \pi_{\mathrm{gen}}: M_{K_{\wt G}}(\wt G)_{\wt E} \cong (T^{\xi}_{K_{Z^\QQ}}\otimes_{E_\Phi}\wt E) \times_{\wt E}(X_U\otimes_E\wt E) \longrightarrow X_U\otimes_E\wt E.
\end{equation*}

Let $\widehat D=(D,g_D)$ be a metrized divisor on $X_U$, and let $Y$ be a cycle whose generic fiber meets $D$ properly. Write $\wt{\mathcal D}$ and $\wt{\mathcal Y}$ for their RSZ realizations. At a finite place $\wt\nu$ where the derived tensor product has finite-length cohomology, define

\begin{equation*}
    \operatorname{Int}_{\wt\nu}(\wt{\mathcal D},\wt{\mathcal Y}) :=\chi\!\left( \widehat{\mathcal M}^{\mathrm{loc}}_{\wt\nu},
    \mathcal O_{\wt{\mathcal D}} \overset{\mathbf L}{\otimes}_{\mathcal O_{\widehat{\mathcal M}^{\mathrm{loc}}_{\wt\nu}}} \mathcal O_{\wt{\mathcal Y}} \right).
\end{equation*}

The right-hand side is computed by the relative Rapoport--Zink description; compare \cite[(9.7)]{MZ}.

The normalized divisor--cycle pairing is

\begin{equation}\label{normalized RSZ intersection}
    \big\langle\widehat D\cdot Y\big\rangle_U :=\frac{1}{\delta_Z[\wt E:E]} \left( \sum_{\wt\nu\nmid\infty} \operatorname{Int}_{\wt\nu}(\wt{\mathcal D},\wt{\mathcal Y})\log N\wt\nu
    +\sum_{\wt\sigma:\wt E\hookrightarrow\CC}g_{D,\wt\sigma}(Y_{\wt\sigma}) \right).
\end{equation}

At infinity we sum over all complex embeddings.  With this convention, a local equation $f=0$ contributes $-\log|f|$.  After conjugate embeddings are grouped, this agrees with the usual squared-norm convention. The projection formula preserves the normalized pairing under a finite extension of the common reflex field. Its level compatibility is given by \eqref{CM-level-pairing}.

From now on, $\mathcal X_U$, $\LL_U$, $\mathcal P$, and $\mathcal Z_t$ denote the normalized RSZ arithmetic objects defined above. Moreover, the PEL moduli problems have natural forgetful morphisms over the common reflex field:

\begin{equation*}
    \mathcal M_{K_{\wt G}}(\wt G)\longrightarrow\mathcal M_{K_{Z^\QQ}}(Z^\QQ), \qquad \mathcal M_{K_{\wt G}}(\wt G)\longrightarrow\mathcal M_{K_{G^\QQ}}(G^\QQ).
\end{equation*}

\subsection{The Hodge bundle and CM cycle}\label{Hodge bundle and CM cycle}

Now we define two important objects on unitary Shimura varieties: the integral \textit{Hodge bundle} and the normalized \textit{CM cycle}.

\subsubsection*{The Hodge bundle}
The tautological line on the generic fibre is equipped with a fixed integral structure. We describe this structure on the local RSZ models, retaining the source filtration of $\mathcal A_0$ in each description.

Suppose first that $E_v/F_v$ is unramified. Let $\mathfrak A_v\subset\mathfrak B_v$ be the local embedding sets and put
\begin{equation*}
    \omega_{\mathcal A,\wt\nu}:=(J_{\mathfrak A_v}\Lie(\mathcal A))_1.
\end{equation*}
The Eisenstein condition gives the inverse Hodge line bundle
\begin{equation}\label{line bundle of modular forms}
    \wt{\mathcal L}_{\wt\nu}^{-1} :=\HHom_{\mathcal O_E\otimes\mathcal O_{\widehat{\mathcal M}^{\mathrm{loc}}_{\wt\nu}}} \bigl(\Fil(\mathcal A_0),\omega_{\mathcal A,\wt\nu}\bigr).
\end{equation}
Here the selected factor is determined by the local signature. We use \cite[Definition~7.1 and Remark~7.2]{LMZ}; the modification functor preserves this line and the polarization maps by \cite[Proposition~7.8]{LMZ}. Thus the same definition applies when $F_v/\QQ_p$ is ramified, with Faltings duals in the relative category. At a split place the two idempotents give the corresponding description.

Suppose that $E_v/F_v$ is ramified. By Assumption~\ref{Assumption}, $v\nmid2$ and $F_v/\QQ_p$ is unramified. Choose a uniformizer $u$ of $E_v$ satisfying $\bar u=-u$, and let $u_S$ be its image on the base. All Lie algebras and filtrations in this description are taken in the relative factor at $v$. Put
\begin{equation*}
    \mathcal I_{\mathcal A,\wt\nu} :=\operatorname{im}\bigl(\iota_{\mathcal A}(u)-u_S: \Lie(\mathcal A)\longrightarrow\Lie(\mathcal A)\bigr).
\end{equation*}
For a $\varpi_{E_v}$-modular lattice, this image is an invertible direct summand of $\Lie(\mathcal A)$; see \cite[Lemmas~2.36--2.37 and~2.41]{LL2}. We define
\begin{equation}\label{ramified line bundle of modular forms}
    \wt{\mathcal L}_{\wt\nu}^{-1} =\HHom_{\mathcal O_E\otimes\mathcal O_{\widehat{\mathcal M}^{\mathrm{loc}}_{\wt\nu}}} \bigl(\Fil(\mathcal A_0),\mathcal I_{\mathcal A,\wt\nu}\bigr).
\end{equation}
This is the relative form of \cite[Definition~2.38]{LL2}, with its source coefficient line retained.

For an almost $\varpi_{E_v}$-modular lattice, so that $n$ is even, we make the following definition on the maximal-level regular model:
\begin{equation}\label{almost modular Hodge extension}
    \wt{\mathcal L}_{\wt\nu}^{-1} :=\HHom_{\mathcal O_E\otimes\mathcal O_{\widehat{\mathcal M}^{\mathrm{loc}}_{\wt\nu}}} \bigl(\Fil(\mathcal A_0), \mathcal I_{\mathcal A,\wt\nu}^{\vee\vee}\bigr).
\end{equation}
The double dual here is taken over the structure sheaf of this regular model. We explain why it is an invertible sheaf with a specified generic identification.

The wedge condition makes every two-by-two minor of $\iota_{\mathcal A}(u)-u_S$ zero. On the open subset where this operator has a unit entry, elementary row and column operations write its matrix as a column times a row, each with a unit coordinate. Its image is therefore an invertible direct summand, including after a nonreduced base change. The complement is the inverse image of the worst-point locus in the local model. In the odd-rank local model this locus has codimension $n+1\ge3$; see \cite[Remark~7.3 and Theorem~7.5]{RSZ1}, where the Hermitian rank is denoted by $n$. The smooth local-model diagram preserves this codimension. The image is a torsion-free sheaf of rank one, so its double dual is a rank-one reflexive sheaf. Regularity makes this double dual invertible. It is the unique extension of the image line from this open subset.

Replacing $u$ by $bu$ leaves the image unchanged. Here $b\in\mathcal O_{F_v}^{\times}$. The construction is equivariant under changes of the polarized trivialization and hence descends through the local-model diagram. On the generic fibre, the original polarizations identify
\begin{equation*}
    \HHom_{E_v}\bigl(\Fil(\mathcal A_0), \mathcal I_{\mathcal A,\wt\nu}\bigr) \cong \pi_{\mathrm{gen}}^*L_U^{-1}.
\end{equation*}
We use this identification in \eqref{ramified line bundle of modular forms} and \eqref{almost modular Hodge extension}. These local extensions, together with \eqref{line bundle of modular forms}, define $\wt{\mathcal L}$ with the polarization-induced identification on the generic fibre.

\begin{remark}
On the local Rapoport--Zink spaces we take the pullback of this line. At a smaller level we pull back the line fixed on the maximal-level model. In the almost modular case its restriction to the open subset where the operator has rank one is the image line in \eqref{ramified line bundle of modular forms}.
\end{remark}

At every archimedean place we transport the tautological metric through the same generic identification. Thus
\begin{equation*}
    \|s_z\|^2=-q(s_z).
\end{equation*}
The normalized metrized line on the unitary Shimura variety is denoted by $\LL_U$, or by $\LL$ when the level is fixed.

\subsubsection*{CM cycle}

Recall the orthogonal decomposition
\begin{equation*}
    \BV=W^\perp(\BA_E)\oplus\BW
\end{equation*}
from \cite[(2.1.1)]{Guo1}, and put $U_\BW=U\cap\U(\BW_{\Af})$.  The natural map of Shimura data gives a finite morphism
\begin{equation*}
    X_{\BW,U_\BW}\longrightarrow X_U.
\end{equation*}
We denote by $P_{\BW,U}\in\Ch_0(X_U)_\QQ$ the image of its fundamental class, normalized to have degree one.

The embedding $\wt G_\BW\hookrightarrow\wt G$ gives the arithmetic CM cycle on the RSZ model.  At a finite place $\wt\nu$ we denote its local realization by $\wt{\mathcal P}_{\wt\nu}$.  After dividing by the toric and reflex-field degrees as in \eqref{normalized RSZ intersection}, we obtain the normalized cycle $\mathcal P_{\BW,U}$.

On a chosen complex lift, the CM cycle is represented by
\begin{equation}\label{CM cycle on unitary Shimura variety}
    [(e),1]\in G(\QQ)\backslash D\times G(\wh{\QQ})/U,
\end{equation}
Here $(e)$ denotes the negative line on a chosen complex lift. We write $P$ and $\mathcal P$ for the generic and arithmetic CM cycles. The cycle $P$ is the pushforward of the complete zero-dimensional Shimura variety, with its stabilizer weights. Formula~\eqref{CM cycle on unitary Shimura variety} specifies a representative, and its torus translates give the complete cycle.

We record its weights on connected components.  Since $n\ge1$, the factor $\operatorname{SU}(V_\iota)$ is noncompact. Strong approximation for $\operatorname{SU}(V)$ identifies the geometric connected components of $X_U$ with
\begin{equation*}
    \U(W)(F)\backslash\U(W)(\BA_f)/\det U.
\end{equation*}
The map from the CM cycle is induced by $U_\BW\subset\det U$. Its fibres have the same weighted cardinality: the quotient of the compact torus levels is fixed, and translation identifies the corresponding stabilizers. The invariant curvature measure of $L$ is transported to the same Haar measure on the conjugate components. Their volumes and their $L$-degrees then have the same weights. Consequently, for every connected component $X_i$,

\begin{equation}\label{CM component weights}
    \deg(P|_{X_i})=\frac{\deg_L(X_i)}{\deg_L(X_U)}.
\end{equation}
In particular, $P-L^n/\deg_L(X_U)$ has degree zero on each component.

Let $U'\subset U$, and write $f:X_{U'}\to X_U$ for the level map. On the corresponding integral models we take $\LL_{U'}=f^*\LL_U$. The normalized CM cycles for the fixed subspace $\BW$ satisfy
\begin{equation}\label{normalized level compatibility}
    \LL_{U'}=f^*\LL_U,\qquad f_*\mathcal P_{\BW,U'}=\mathcal P_{\BW,U}.
\end{equation}
The same equality holds for their closures by proper pushforward. The projection formula applied to \eqref{normalized level compatibility} gives
\begin{equation}\label{CM-level-pairing}
    (f^*\widehat D)\cdot\mathcal P_{\BW,U'} =\widehat D\cdot\mathcal P_{\BW,U}.
\end{equation}
For the pullback of the whole CM cycle, the finite flat level maps of \cite[Theorem~4.4.4]{Qiu} give
\begin{equation*}
    f^*\mathcal P_{\BW,U} =\sum_{h\in U_\BW\backslash U/U'} \frac{d_{h^{-1}\BW,U'}}{d_{\BW,U}} \mathcal P_{h^{-1}\BW,U'},
\end{equation*}
where $d_{\BW,U}$ is the degree of the CM cycle with its stabilizer weights; see \cite[(5.4)]{Qiu}. These formulas are interpreted on a common neat cover and then divided by its degree.

\section{Unitary Shimura curves}\label{Unitary Shimura curve}

In this section, we focus on the one-dimensional case of the unitary Shimura variety. We first extract a special class of Shimura curves from unitary Shimura varieties of general dimension, which serve as important auxiliary objects for computing the modular height of CM points. We then compare the unitary Shimura curve with the quaternionic Shimura curve in \cite{YZZ2,YZ1,Yuan1} through the corresponding auxiliary Shimura curves and morphisms of their integral models.

\subsection{Special curves in unitary Shimura varieties}

To compute the modular height of the CM cycle, we choose a two-dimensional incoherent Hermitian subspace $\BV_1\subset\BV$ containing $\BW$.  Write
\begin{equation*}
    \BV=\BV_1\oplus V_1^\perp(\BA_E), \qquad \BV_1=\BW\oplus W_1^\perp(\BA_E),
\end{equation*}
and let $V_1$ be the nearby coherent space at $\iota$.  With $U_{\BV_1}=U\cap\U(\BV_{1,\BA_f})$, there are morphisms on generic fibers
\begin{equation*}
    X_{\BW,U_\BW}\longrightarrow C_U:=X_{\BV_1,U_{\BV_1}} \longrightarrow X_U.
\end{equation*}
Thus $P_{\BW,U}\subset C_U\subset X_U$.

The corresponding arithmetic curve is obtained from the RSZ datum attached to $V_1$.  If the orthogonal inclusion of lattices extends to the local lattice chains, it gives a morphism of the corresponding RSZ local models.  We denote the normalized arithmetic curve by $\mathcal C_U$.

At the places where $\Lambda_1=\Lambda\cap V_1$ has the required local type, the local Hodge line bundle is compatible with restriction.  If an auxiliary choice of $V_1$ is needed at finitely many places, we use the following standard argument. Roughly speaking, this is a ``perturbation” method that allows more relaxed conditions at finitely many places.

Let $\LL_1$ denote the Hodge line bundle on $\mathcal C_U$, defined by the local formulas of Section~\ref{Hodge bundle and CM cycle}. For convenience, we call the subspace $\BW\subset\BV$ the ``CM line".

\begin{proposition}
Let $\BW\subset\BV$ be the fixed compatible CM line. There is an auxiliary unitary Shimura curve containing this CM line and a finite set $S$ of places such that
\begin{equation}\label{aux-CM}
    \LL\cdot\mathcal P -2\widehat{\deg}_F(\LL_B|_{\mathcal P_{U_B}}) -\log d_{E/F} \in\sum_{v\in S}\QQ\log N_v.
\end{equation}
Here $\LL_B$ and $\mathcal P_{U_B}$ are the Hodge line bundle and normalized CM cycle on the associated quaternionic Shimura curve with maximal-order level, recalled in the next subsection. The set $S$ is disjoint from the ramified places of $E/F$. Given a finite set of rational primes, the auxiliary choices can also be made with $S$ supported away from those primes.
\end{proposition}

\begin{proof}
First, we choose the compatible rank-two summands at every place ramified in $E/F$, at the places of the other fixed local conditions, and above the given rational primes. In each summand choose the positive direction orthogonal to $\BW_v$. Its norm class and integral lattice-chain embedding are constant on a sufficiently small open neighbourhood. Weak approximation in $W^\perp$ then gives a global positive vector $x$ in these neighbourhoods, which defines a subspace $V_1=W\oplus Ex$. By definition, outside a finite set the induced lattice splits integrally and the positive line has unit norm. At the remaining places, we choose a maximal rank-two lattice, allowing the inert almost self-dual type. Let $S$ contain the places where these additional choices change the lattice or its Hodge comparison. By construction, $S$ avoids the prescribed places. On the associated quaternionic curve (which we refer to the next subsection), we choose maximal orders containing $\mathcal O_E$.

The comparison keeps the Hermitian coefficient line $C=Ex$, with $\|x\|_\sigma^2=q(x)_\sigma$. Its finite lattices are those of the marked positive CM factor in the Morita comparison. At a ramified place, the unit-norm direction of Proposition~\ref{comparison of Hodge bundles on curves} gives
\begin{equation*}
    C_v=\mathcal O_{E_v}e'_v,\qquad \Lambda_{1,v}\cap E_ve'_v=uC_v.
\end{equation*}
At the other places outside $S$, $C_v$ is the norm-unit positive-factor lattice. Thus $C$ is self-dual outside $S$, while the ramified ambient lattice factor $u$ occurs in the divisor of the comparison map. Write $C_\nu=\mathcal O_{E_\nu}a_\nu x$ and $k_\nu=\operatorname{ord}_\nu(a_\nu)$. The product formula gives

\begin{equation}\label{auxiliary coefficient line degree}
    \begin{aligned}
        &-\sum_{\nu\nmid\infty}k_\nu\log N\nu -\frac12\sum_{\sigma:E\hookrightarrow\CC}\log q(x)_\sigma\\
        &\qquad=-\frac12\sum_{\nu\nmid\infty} \bigl(k_\nu+k_{\bar\nu}+\operatorname{ord}_\nu q(x)\bigr) \log N\nu.
    \end{aligned}
\end{equation}

The rational wedge comparison of Proposition~\ref{comparison of Hodge bundles on curves} has source $\mathcal L_1|_{\mathcal P}\otimes C^{-1}$. Write it as $\alpha:\mathcal L_1|_{\mathcal P}\otimes C^{-1}\dashrightarrow\mathcal L_{B,E}|_{\mathcal P}$. For $\overline\Delta=2\LL_1-\LL_{B,E}$, one has
\begin{equation*}
    \overline\Delta|_{\mathcal P} =\LL_1|_{\mathcal P}+\overline C -\widehat{\operatorname{div}}(\alpha).
\end{equation*}
Replacing $C_v$ by $uC_v$ reduces both its local arithmetic degree and the local degree of $\operatorname{div}(\alpha)$ by $\log N\nu$. Hence their difference is independent of this lattice convention. Outside $S$, self-duality of $C$ makes the summands in \eqref{auxiliary coefficient line degree} zero. Changing a generator by $a_\nu$ multiplies the wedge map by $a_\nu$ and its Hermitian adjoint by $\bar a_\nu$. Their archimedean norms are those of the Hermitian line $Ex$ used in \eqref{auxiliary coefficient line degree}. At the places in $S$, the remaining differences are the valuations of the rational comparison maps on the normalized zero-dimensional Shimura varieties. Their degrees therefore lie in $\sum_{v\in S}\QQ\log N_v$.

At every place outside $S$, Proposition~\ref{restriction of Hodge bundle on curve} compares the ambient and curve line bundles integrally. The local calculations in Proposition~\ref{comparison of Hodge bundles on curves} then compare the curve line with $\LL_B$. The ramified places contribute $\log d_{E/F}$. Together with \eqref{auxiliary coefficient line degree} and the finite differences at $S$, this proves \eqref{aux-CM}. The quaternion algebra is split at every finite place outside $S$, so $\log d_\BB\in\sum_{v\in S}\QQ\log N_v$. The construction keeps the CM line $\BW$ and its normalized CM cycle fixed while varying the auxiliary positive direction and lattice.
\end{proof}

\begin{proposition}\label{restriction of Hodge bundle on curve}
Suppose that the orthogonal inclusion of the rank-two lattice containing the CM line extends to the local RSZ lattice chains. The generic tautological identification extends to an integral isometry
\begin{equation*}
    \LL|_{\mathcal C_U}\cong\LL_1.
\end{equation*}
This includes both ramified lattice types in Assumption~\ref{Assumption}.
\end{proposition}

\begin{proof}
At an unramified place, the orthogonal product decomposition identifies the selected factor $(J_{\mathfrak A_v}\Lie\mathcal A)_1$ with its rank-two counterpart. The selected factor comes from the rank-two summand, and $\Fil\mathcal A_0$ is unchanged. Formula~\eqref{line bundle of modular forms} therefore gives an integral identification. The relative modification preserves it by \cite[Proposition~7.8]{LMZ}. At a split place the same argument uses the two idempotents.

At a ramified place write $\Lambda_v=\Lambda_{1,v}\oplus\Lambda_{1,v}^\perp$, where $\Lambda_{1,v}$ is $\varpi_{E_v}$-modular. In the Serre tensor description of the rank-two factor, we use the basis $1,u$. The operator on its Lie algebra is
\begin{equation}\label{ramified restriction matrix}
    \begin{pmatrix}
        -u_S&u_S^2\\
        1&-u_S
    \end{pmatrix}
    =\binom{-u_S}{1}(1,-u_S).
\end{equation}
The additional positive factors are $(\mathcal E^2)^{(n-1)/2}$ when $n$ is odd, and $(\mathcal E^2)^{(n-2)/2}\times\mathcal E$ when $n$ is even, with the polarizations of \cite[Section~6.1.2]{Qiu}. On their Lie algebras $\iota(u)-u_S$ is zero. Thus the ambient operator restricts to the block matrix consisting of \eqref{ramified restriction matrix} and a zero block. Its image is
\begin{equation*}
    \mathcal O\,(-u_S,1,0,\ldots,0)^t.
\end{equation*}
The unit coordinate shows that the curve lies in the rank-one open subset. In the almost modular case, the arithmetic line bundle defined in \eqref{almost modular Hodge extension} is already the image line bundle on this open subset. Then the pull-back gives the line bundle on the curve. Applying $\HHom(\Fil\mathcal A_0,\cdot)$ proves the integral identification with the same filtration. Note that this construction is equivariant for the framing and level actions and hence descends from the local product models.

The polarizations identify these maps on the generic fibre with the tautological inclusion of the negative line. Both metrics are the restriction of $-q$, so the identification is an isometry at infinity. The normalization by toric and reflex-field degrees preserves it.
\end{proof}

\subsection{Relation with quaternionic Shimura curves}

We compare the rank-two unitary datum with its quaternionic realization. The resulting PEL morphisms will identify the generic Kodaira--Spencer map and its integral corrections.

\subsubsection*{Definition of quaternionic Shimura curves}

First, let us briefly review the definition of quaternionic Shimura curves; for a detailed discussion, we refer to \cite[Section~1.2.1]{YZZ2}. Let $\Sigma$ be a finite set of places of $F$ containing all the archimedean places and having odd cardinality. Let $\BB$ be the totally definite incoherent quaternion algebra over $\BA_F$ with ramification set $\Sigma$. Let $U_B=\mathcal O_{\BB_f}^{\times}$ be the maximal open compact subgroup given by the multiplicative group of the maximal order.

The corresponding smooth projective curve $C_{U_B}$ is defined over $F$ and has complex uniformization

\begin{equation*}
    C_{U_B,\iota}(\CC)=B(\iota)^\times\backslash(\mathcal{H}^\pm\times\BB^\times_f/U_B),
\end{equation*}
Here $B(\iota)$ is the nearby quaternion algebra over $F$ with ramification set $\Sigma\setminus\{\iota\}$. Since $F\ne\QQ$, the curve is proper. We use its smooth projective coarse curve and the corresponding stack weights in the arithmetic normalization.

The integral model $\mathcal C_{U_B}$ over $\mathcal O_F$ is smooth away from the finite discriminant of $B(\iota)$; see \cite[Section~4]{YZ1}. For the standard rank-two unitary datum this discriminant is trivial by Remark~\ref{split remark of quaternion}. For the relaxed curve datum, its support is $\Sigma_f$ and its norm is
\begin{equation*}
    d_\BB=\prod_{v\in\Sigma_f}N_v.
\end{equation*}
We denote by $\LL_B$ the integral Hodge line bundle with Yuan's metric and normalization. On a common PEL cover, Proposition~\ref{comparison of Hodge bundles on curves} compares its pullback with $\LL_1$.

Let $P_{U_B}$ be the normalized degree-one CM cycle of \cite{YZ1,Yuan1}, obtained from the full zero-dimensional Shimura variety with its stabilizer weights. On a chosen complex lift its negative point is
\begin{equation*}
    [(e),1]\in B(\iota)^\times\backslash(\mathcal H^\pm\times\BB_f^\times/U_B).
\end{equation*}
Here $e\in\mathcal H$ is the fixed point of $E^\times$ under the embedding $E\hookrightarrow B(\iota)$. The normalized closure of the complete CM cycle is denoted by $\mathcal P_{U_B}$.

\subsubsection*{Quaternion algebras and Hermitian spaces}

We first recall the algebraic correspondence between quaternion embeddings and Hermitian planes, and then compare their integral lattices.

\begin{proposition}\label{isomorphism between quaternion and unitary space}
There is a bijection
\begin{equation*}
    \begin{aligned}
        &\{\mathrm{iso.\ class\ of\ incoherent\ totally\ definite\ quaternion\ embeddings\ of } \BA_E\} \\
        &\cong\{\mathrm{iso.\ class\ of\ incoherent\ totally\ definite\ Hermitian\ space\ over\ }\BA_E\ \mathrm{of\ dim.\ 2}\}.
    \end{aligned}
\end{equation*}

Here a quaternion embedding is a pair consisting of $\BB$ and an embedding $\BA_E\hookrightarrow\BB$ of $\BA_F$-algebras. Moreover, if $\BB$ and the Hermitian plane $\BV_1$ are related in this way, then
\begin{equation*}
    \epsilon_v(\BB)=(-1,D_v)_v\cdot\epsilon_v(\BV_1),
\end{equation*}
where $\epsilon_v(\BB)$ and $\epsilon_v(\BV_1)$ are the invariants of the quaternion algebra and the Hermitian plane, respectively, and $D_v\in F_v^\times/F_v^{\times2}$ is the discriminant class of $E_v/F_v$.
\end{proposition}
\begin{proof}

We recall the construction in \cite[Proposition~3.1.4]{How}. It is compatible with passage to the nearby coherent spaces, so we describe it in the coherent case.

Given a quaternion embedding $E\hookrightarrow B$, write

\begin{equation*}
    B=E\oplus Ej,
\end{equation*}
where $j\in B^\times$ is the unique element up to $E^\times$-scaling satisfying $ej=j\overline{e}$ for all $e\in E$. Then the pairing $\langle\cdot,\cdot\rangle:B\times B\rightarrow E$ defined by
\begin{equation*}
    \langle b_1,b_2\rangle=\pi(b_1 \overline{b}_2)
\end{equation*}
is a non-degenerate Hermitian form on $B$. Here $\pi:B\rightarrow E$ is the projection to the first component under the decomposition. Note that under this definition, the reduced norm of the quaternion algebra agrees with the  Hermitian norm.

Conversely, the Hermitian plane represents $1$ at every finite place: an isotropic plane represents every scalar, and an anisotropic plane has a scalar multiple of a quaternionic reduced norm, which is surjective locally. At infinity the positive or indefinite plane also represents $1$. The Hasse principle for its underlying quadratic form therefore gives a vector $v\in V$ of norm $1$. Write

\begin{equation*}
    V=Ev\oplus (Ev)^\perp.
\end{equation*}
Then the multiplication structure is defined by
\begin{equation*}
    (\alpha v+w)(\alpha' v+w')=(\alpha\alpha'-\langle w,w'\rangle)v+(\alpha w'+\overline{\alpha}'w),
\end{equation*}
where $\alpha,\alpha'\in E$ and $w,w'\in(Ev)^\perp$.

Note that
\begin{equation*}
    U(B)\cong (B^\times \times E^\times)^1/ F^\times,
\end{equation*}
where the superscript 1 means that it consists of elements $(b,e)$ such that $\nu(b)=\Nm_{E/F}(e)$ and $F^\times$ embeds diagonally. Similarly,
\begin{equation*}
    \GU(B)\cong(B^\times \times E^\times)/ F^\times
\end{equation*}
and the similitude factor is given by $(b,e)\mapsto\nu(b)\Nm_{E/F}(e)^{-1}$. Here $\nu$ is the reduced norm of $B$. The correspondence of local invariants follows from the reduced norm form.

\end{proof}

\begin{example}

Suppose $V$ has an orthonormal basis $e_1,e_2$ and write

\begin{equation*}
    (a,b):=ae_1+be_2,\quad a,b\in E
\end{equation*}
to denote the vector in $V$. Then  the multiplication in the quaternion algebra is defined by
\begin{equation*}
    (a,b)\times(c,d)=(ac-b\overline{d},ad+b\overline{c}),
\end{equation*}
and the main involution is defined by
\begin{equation*}
    \overline{(a,b)}=(\overline{a},-b).
\end{equation*}
\end{example}

\begin{remark}\label{split remark of quaternion}
For the standard rank-two lattice of Assumption~\ref{Assumption}, the quaternion algebra $\BB$ corresponding to $\BV_1$ is split at every finite place.  If we use the relaxed curve datum of Section~\ref{curve introduction}, then $\BB_v$ is division exactly for $v\in\Sigma_f$; these are the inert places at which $\Lambda_{1,v}$ is almost self-dual. This fact will be proved rigorously below.
\end{remark}

\subsubsection*{Maximal orders and Hermitian lattices}
To compare the level groups, we use the lattice--order correspondence under Proposition~\ref{isomorphism between quaternion and unitary space}.  The required local statement is the following.

\begin{proposition}\label{relation of lattice and order}

Let $\BB$ correspond to $\BV_1$ under Proposition~\ref{isomorphism between quaternion and unitary space}. If $v$ is unramified in $E/F$, the self-dual lattice $\Lambda_{1,v}$ corresponds to the maximal order $\mathcal O_{\BB_v}$. If $v$ is ramified in $E/F$, the $\varpi_{E_v}$-modular lattice corresponds to
\begin{equation*}
    \mathcal O_{\BB_v}^{\vee} =\varpi_{E_v}\mathcal O_{\BB_v},
\end{equation*}
a principal right ideal of the maximal order. Its Hermitian dual is $\mathcal O_{\BB_v}$.

\end{proposition}

\begin{proof}

Use the decomposition from Proposition~\ref{isomorphism between quaternion and unitary space}:

\begin{equation*}
    \BB_v=E_v\oplus E_vj_v,
\end{equation*}
and we may further assume $v(q(j_v))=0$. Then, in the case when $v$ is unramified in $E/F$, the self-dual lattice $\Lambda_{1,v}$ is identified with $\mathcal{O}_{E_v}\oplus \mathcal{O}_{E_v}j_v$, while it is well-known that in this case
\begin{equation*}
    \mathcal{O}_{\BB_v}=\mathcal{O}_{E_v}\oplus \mathcal{O}_{E_v}j_v.
\end{equation*}

At a place ramified in $E/F$, the quaternion algebra $\BB_v$ is split by Remark~\ref{split remark of quaternion}. Choose the embedding

\begin{equation*}
    x\mapsto\matrixx{x}{}{}{x},\quad \varpi_{E_v}\mapsto\matrixx{}{-p_v}{1}{},
\end{equation*}
where $x$ is any element in $F_v$ and $p_v$ is a uniformizer of $F_v$. Then by our definition of $j_v$, we can further assume that
\begin{equation*}
    j_v\mapsto\matrixx{1}{}{}{-1}.
\end{equation*}
Note that up to isomorphism, $\Lambda_{1,v}$ is generated over $\mathcal O_{E_v}$ by $\varpi_{E_v}(1+j_v)$ and $(1-j_v)$.  In the above matrix realization this gives
\begin{equation*}
    \Lambda_{1,v} =\left\{\matrixx{p_vx}{p_vy}{z}{w}:x,y,z,w\in\mathcal O_{F_v}\right\} \subset M_2(\mathcal O_{F_v})=\mathcal O_{\BB_v}.
\end{equation*}

Similarly, the dual lattice $\Lambda_{1,v}^\vee$ is generated over $\mathcal O_{E_v}$ by {$(1+j_v)$ and $\varpi_{E_v}^{-1}(1-j_v)$}, and hence
\begin{equation*}
    \Lambda_{1,v}^\vee =\left\{\matrixx{x}{y}{z}{w}:x,y,z,w\in\mathcal O_{F_v}\right\} =M_2(\mathcal O_{F_v})=\mathcal O_{\BB_v}.
\end{equation*}

Thus $\Lambda_{1,v}=\varpi_{E_v}M_2(\mathcal O_{F_v})$ is a principal right ideal of the maximal order, and its Hermitian dual is that maximal order. Left multiplication by $\varpi_{E_v}\in E_v$ is an integral right-module isomorphism from $M_2(\mathcal O_{F_v})$ to $\Lambda_{1,v}$. This identifies the same maximal-order deformation factor under the PEL comparison.

\end{proof}

The division quaternion case gives the remaining lattice types. At an inert place the corresponding Hermitian plane has invariant $-1$ and carries an almost self-dual lattice. At a ramified place the anisotropic plane carries a self-dual lattice. The following proposition identifies their maximal orders.

\begin{proposition}\label{relation of lattice and order nonsplit case}
Suppose $\BB$ is the quaternion algebra corresponding to $\BV_1$ under the bijection in Proposition \ref{isomorphism between quaternion and unitary space}, such that $\BB_v$ is a division quaternion algebra over $F_v$ at a finite place $v$. If $v$ is unramified in $E/F$, the almost self-dual lattice $\Lambda_{1,v}\subset\BV_{1,v}$ corresponds to the maximal order $\mathcal{O}_{\BB_v}$; if $v$ is ramified in $E/F$, the self-dual lattice $\Lambda_{1,v}$ also corresponds to the maximal order $\mathcal{O}_{\BB_v}$.
\end{proposition}
\begin{proof}
As we assume $\BB_v$ is a division algebra, it is sufficient to check that $\Lambda_{1,v}$ consists of vectors $x$ such that $v(q(x))\ge 0$, since the  Hermitian norm  and the reduced quaternionic norm are compatible.

For the inert case, note that by definition, the almost self-dual lattice $\Lambda_{1,v}$ is generated by an orthogonal basis $\{e_1,e_2\}$ such that
\begin{equation*}
    v(q(e_1))=0,\ v(q(e_2))=1.
\end{equation*}

The two summands have norm valuations of opposite parity. The valuation of their sum is therefore the smaller of the two valuations, so the total norm is integral exactly when both coordinates are integral. This identifies the lattice with the integral-norm elements of the division algebra.

For the ramified case, the division condition is equivalent to the associated Hermitian plane being anisotropic.  Since the residue characteristic is odd under Assumption~\ref{Assumption}, we may diagonalize it in the form
\begin{equation*}
    q(ae_1+be_2)=\Nm(a)+u\Nm(b), \qquad u\in\mathcal O_{F_v}^{\times},\quad -u\notin\Nm(E_v^{\times}).
\end{equation*}
We claim that
\begin{equation}\label{anisotropic ramified norm valuation}
    v\bigl(\Nm(a)+u\Nm(b)\bigr) =\min\{v(\Nm(a)),v(\Nm(b))\}.
\end{equation}
When the two valuations differ, the valuation of their sum is the smaller one. When they agree, divide by their common power of the uniformizer. Then the residue of the sum is nonzero because $-u\notin\Nm(E_v^\times)$, by the norm criterion for a ramified quadratic extension of odd residue characteristic. This proves \eqref{anisotropic ramified norm valuation}.  Consequently $q(ae_1+be_2)$ is integral if and only if $a,b\in\mathcal O_{E_v}$.  The integral-norm lattice is therefore the standard self-dual lattice $\mathcal O_{E_v}e_1\oplus\mathcal O_{E_v}e_2$, and under the Hermitian--quaternion identification it is the inverse image of
\begin{equation*}
    \{x\in\BB_v:v(\operatorname{Nrd}(x))\ge0\}=\mathcal O_{\BB_v}.
\end{equation*}
This proves the ramified division case.
\end{proof}

Propositions~\ref{relation of lattice and order} and~\ref{relation of lattice and order nonsplit case} give the lattice--order correspondences summarized in Figure~\ref{Hermitian lattice and maximal order}.

\begin{figure}[ht]
\centering
\begin{center}
\begin{equation*}
    \mathrm{maximal\ order}\longleftrightarrow\left\{
    \begin{aligned}
        \nonumber &\mathrm{self\ dual\ lattice} \ \ \ (v\ \mathrm{split});\\
        &\mathrm{self\ dual\ lattice} \ \ \ (v\ \mathrm{inert},\ \epsilon(\BB_v)=1);\\
        &\mathrm{almost\ self\ dual\ lattice} \ \ \ (v\ \mathrm{inert},\ \epsilon(\BB_v)=-1);\\
        &\mathrm{dual\ of}\ \varpi_{E_v}\mathrm{-modular\ lattice} \ \ \ (v\ \mathrm{ramified},\ \epsilon(\BB_v)=1);\\
        &\mathrm{self\ dual\ lattice} \ \ \ (v\ \mathrm{ramified},\ \epsilon(\BB_v)=-1).
    \end{aligned}
    \right.
\end{equation*}
\end{center}
\caption{Hermitian lattice and maximal order} \label{Hermitian lattice and maximal order}

\end{figure}

\subsubsection*{Auxiliary Shimura curves}

We now compare the unitary Shimura curve with the quaternionic curve through PEL-type auxiliary Shimura curves.  All varieties in this paragraph are first defined over their own reflex fields and are then base changed to a common finite extension before the maps are formed.

Put
\begin{equation*}
    G_1^\QQ:=\BG_m\times_{\Res_{F/\QQ}\BG_m}\Res_{F/\QQ}\GU(V_1), \qquad \wt G_1:=Z^\QQ\times_{\BG_m}G_1^\QQ.
\end{equation*}
Thus the similitude factor in $G_1^\QQ$ is required to be rational.  Let $r_1$ be the rank-two strict fake Drinfeld generalized CM type attached to $V_1$.  The reflex fields of the unitary-similitude and RSZ data are, respectively,
\begin{equation*}
    E_{r_1}, \qquad \wt E_1:=E_\Phi E_{r_1}.
\end{equation*}

The fields $E_{r_1}$ and $\wt E_1$ are the reflex fields specified in \cite[Example~2.3]{RSZ2}. We form all comparison morphisms after embedding them in a common finite extension.

Let $K_{\wt G_1}=K_{Z^\QQ}\times U_{\Lambda_1}$.  The product decomposition of the RSZ group gives, over $\wt E_1$,
\begin{equation*}
    \wt C_U:=S(\wt G_1,\{h_{\wt G_1}\})_{K_{\wt G_1}} \cong \big(T_{K_{Z^\QQ}}\otimes_{E_\Phi}\wt E_1\big) \times_{\wt E_1} \big(C_U\otimes_E\wt E_1\big).
\end{equation*}
The corresponding RSZ integral model is denoted by $\wt{\mathcal C}_U$.  At a ramified place the rank-two lattice is $\varpi_{E_v}$-modular, and this notation refers to its completed RSZ model and relative local description.  The map to the unitary curve that is used in the normalization is the generic projection
\begin{equation*}
    \pi_{1,\mathrm{gen}}:\wt C_U\longrightarrow C_U\otimes_E\wt E_1;
\end{equation*}
For the arithmetic comparison we use the induced morphisms between the corresponding PEL/RSZ models after base change.  The RSZ Hodge line bundle $\wt{\mathcal L}_1$ is the integral extension of the automorphic line described in Section~\ref{Hodge bundle and CM cycle}.

Via the quaternionic realization of $V_1$ from Propositions~\ref{isomorphism between quaternion and unitary space}, \ref{relation of lattice and order}, and~\ref{relation of lattice and order nonsplit case}, the group denoted $G'$ in \cite[Section~3]{YZ1} identifies with $G_1^\QQ$.  With $U'\subset G'(\Af)$ chosen compatibly, its Shimura curve is
\begin{equation*}
    C'_{U'}(\CC)=G'(\QQ)\backslash\mathcal H^\pm\times G'(\Af)/U'.
\end{equation*}
Its reflex field is $E_{r_1}$. The PEL moduli problem gives an integral model $\mathcal C'_{U'}$ over its reflex field (with the corresponding relative local model at bad places) and a metrized automorphic line $\LL'$.  The projection $\wt G_1\to G_1^\QQ$ induces, after base change to $\wt E_1$, a morphism of PEL integral models
\begin{equation*}
    \wt{\mathcal C}_U\longrightarrow \mathcal C'_{U'}\otimes_{\mathcal O_{E_{r_1}}}\mathcal O_{\wt E_1}.
\end{equation*}

The pullback of the filtered PEL data is compatible with the source $\mathcal A_0$, its polarization, and Faltings duality. We keep the factor $\Fil(\mathcal A_0)$ in the Hom description of $\wt{\mathcal L}_1$. The resulting Hodge bundle comparison is computed in Proposition~\ref{comparison of Hodge bundles on curves}.

\begin{remark}\label{Kodaira-Spencer remark}
On the common generic fibre the Kodaira--Spencer relation gives $c_1(L_B)=2c_1(L)$.  For metrized arithmetic lines the comparison also contains the determinant character.  Proposition~\ref{comparison of Hodge bundles on curves} computes this factor, including its integral extension and metric.
\end{remark}

Finally, let $G''$ and $C''_{U''}$ be the similitude group and Shimura curve of \cite[Section~5]{YZ1}, and let $E''$ be its reflex field.  Its complex uniformization is
\begin{equation*}
    C''_{U'',\iota}(\CC)=G''(\QQ)\backslash\mathcal H^\pm\times G''(\Af)/U''.
\end{equation*}
Let $E^{\mathrm{aux}}$ be a finite field extension containing $\wt E_1$, $E''$, and the reflex field of the quaternionic curve.  The morphisms constructed in \cite[Section~5.3]{YZ1} extend, after base change to $\mathcal O_{E^{\mathrm{aux}}}$ and with the corresponding local modifications, to morphisms of PEL models
\begin{equation*}
    \mathcal C_{U_B}\longrightarrow\mathcal C''_{U''}, \qquad \mathcal C'_{U'}\longrightarrow\mathcal C''_{U''}.
\end{equation*}

Over the common field $E^{\mathrm{aux}}$ the auxiliary curves fit into the following diagram.  The solid arrows are morphisms of PEL models, and the dashed arrow is the generic projection to the unitary curve.

\begin{figure}[ht]
\centering
\begin{tikzpicture}[squarednode/.style={rectangle, draw=red!60, fill=red!5, very thick, minimum size=5mm},squarednode2/.style={rectangle, draw=green!60, fill=green!5, very thick, minimum size=5mm}]
\node[squarednode]  (1) {$C_U$};
\node[squarednode2] (2) [above right=1cm of 1] {$\wt{\mathcal C}_U$};
\node[squarednode2] (3) [below right=1cm of 2] {$\mathcal C'_{U'}$};
\node[squarednode]  (4) [below right=1cm of 3] {$\mathcal C''_{U''}$};
\node[squarednode]  (5) [above right=1cm of 4] {$\mathcal C_{U_B}$};
\draw[->,dashed] (2.south west) -- (1.north east);
\draw[->] (2.south east) -- (3.north west);
\draw[->] (3.south east) -- (4.north west);
\draw[->] (5.south west) -- (4.north east);
\end{tikzpicture}
\caption{Comparison of the auxiliary curves over a common reflex field} \label{relation of curves}
\end{figure}

The CM cycles are compatible under the solid PEL morphisms, and the generic CM cycle on $C_U$ is obtained from the RSZ one by the normalized projection.  The precise comparison of the two metrized automorphic lines is the following.

\begin{proposition}\label{comparison of Hodge bundles on curves}
On the common curve over $E^{\mathrm{aux}}$, let $\LL_{B,E}$ be the normalized pullback of Yuan's quaternionic Hodge line bundle. Put
\begin{equation*}
    \overline\Delta=2\LL_1-\LL_{B,E}.
\end{equation*}
For the curve data of this section, including the inert almost self-dual places in $\Sigma_f$, one has
\begin{equation}\label{intersection of determinant difference line}
    \begin{aligned}
        \frac{\overline\Delta^2}{\deg L_1} &=-\sum_{v\in\Sigma_f}\frac{2(N_v+1)}{N_v-1}\log N_v,\\
        \widehat{\deg}_E(\overline\Delta|_{\mathcal P}) &=\LL_1\cdot\mathcal P+\log d_{E/F},\\
        \frac{\LL_1\cdot\overline\Delta}{\deg L_1} &=\LL_1\cdot\mathcal P+\log d_{E/F} -\sum_{v\in\Sigma_f}\frac{2N_v}{N_v-1}\log N_v.
    \end{aligned}
\end{equation}
Consequently,
\begin{equation}\label{exact point Hodge comparison}
    \LL_1\cdot\mathcal P =2\widehat{\deg}_F(\LL_B|_{\mathcal P_{U_B}})+\log d_{E/F},
\end{equation}
and
\begin{equation}\label{exact curve Hodge comparison}
    \begin{aligned}
        h_{\LL_1}(C_U) ={}&2h_{\LL_B}^{\mathrm{Yuan}}(C_{U_B}) +\LL_1\cdot\mathcal P+\log d_{E/F}\\
        &-\sum_{v\in\Sigma_f}\frac{3N_v-1}{2(N_v-1)}\log N_v.
    \end{aligned}
\end{equation}
\end{proposition}

\begin{proof}
The strategy of our proof is as follows. We first specify the generic comparison and its metric, then compute its integral divisor. We determine this divisor place by place. If the special fibre is irreducible, it is enough to compute the valuation of the comparison at one convenient integral point. If the special fibre has several irreducible components, we compute the degrees of the two line bundles on each component; equivalently, these are the intersection numbers of the vertical divisor with the components. Since the intersection matrix of the special fibre has kernel generated by the whole fibre, these degrees determine the vertical correction up to a class pulled back from the base. The remaining base class is then fixed by restricting the comparison to the CM cycle. The vertical correction is computed through the moduli interpretation of the two automorphic lines.

Note that all comparisons are made through the PEL morphisms in Figure~\ref{relation of curves}.

Let $\mathcal H$ be the rank-two automorphic Hermitian bundle. The cotangent space of a line in a plane gives
\begin{equation}\label{generic determinant KS}
    \Omega^1_{C_U}\cong\Hom(\mathcal H/L_1,L_1)\cong L_1^{\otimes2}\otimes(\det\mathcal H)^{-1}.
\end{equation}
The polarization pairing identifies the obstruction section associated with a special homomorphism $x$ with
\begin{equation*}
    s\longmapsto\lambda_0^{-1}x^\vee\lambda(s)=\langle s,x\rangle.
\end{equation*}
Our Hermitian pairing is linear in the first variable. Differentiating this section in the positive normal direction to the CM line gives the wedge map in \eqref{generic determinant KS}. The filtered tensor construction in \cite[Section~5.2, Propositions~5.3--5.4 and Corollary~5.5]{YZ1}  carries this map to the quaternionic Kodaira--Spencer map. Thus the generic fibre of $\Delta$ is $\det\mathcal H$, with this specified identification.

At the distinguished complex place choose an orthogonal basis $e',e$ of norms $1,-1$ and put $s_z=ze'+e$ on the unit disk. Up to a sign, the map is
\begin{equation*}
    \frac{s_z^{\otimes2}}{e'\wedge e}\longmapsto dz.
\end{equation*}
For $z=(\tau-i)/(\tau+i)$ one has
\begin{equation*}
    \|s_z\|^2=1-|z|^2 =2\operatorname{Im}(\tau)\left|\frac{dz}{d\tau}\right|.
\end{equation*}
This is precisely the metric on Yuan's cotangent line. The generic comparison is therefore an isometry at every archimedean place.

We next identify the integral pullback of Yuan's line. At a ramified place the Serre tensor description of \cite[Section~6.1.2]{Qiu} and the split quaternionic Morita equivalence recover the same relative height-two group. The induced map on its deformation space is then the original cotangent map. At an inert almost self-dual place, the functor of \cite[Theorem~1.2]{KR} retains the underlying formal group and recovers its quaternionic action. Its cotangent line is the relative dualizing line on the Drinfeld model, and the relative modification in \cite[Proposition~7.8]{LMZ} preserves the selected Lie factor and the polarization maps. These descriptions then identify the target with the pullback of the same integral line $\LL_B$ used in \eqref{exact point Hodge comparison}. On the common PEL base, the comparison is therefore reduced to the map from $\Fil\mathcal A_0$ to the selected Lie quotient. The additional toric tensor factors appear on both sides with their duals and hence cancel in the determinant.

Suppose first that $v$ is ramified in $E/F$. Write $u^2=\varpi_{E_v}$ and $\bar u=-u$. In the Serre description let $c$ correspond to the special homomorphism $ue'_v$, where $q(e'_v)$ is a unit. The Hermitian norm is $-2\pi\operatorname{Nrd}$, and hence $\operatorname{Nrd}(c)=q(e'_v)/2$ is a unit. Since $v\nmid2$, $c$ is an integral unit. On the filtration of $\Fil\mathcal A_0$, the special homomorphism $e'_v$ is

\begin{equation}\label{ramified fixed comparison}
    \binom{c}{-u_S^{-1}c} =-u_S^{-1}\binom{-u_S}{1}c.
\end{equation}

Let $G$ be a relative height-two group,
\begin{equation*}
    \kappa_G:\Fil G\otimes \Lie G^\vee\longrightarrow\omega_G
\end{equation*}
be the Kodaira--Spencer map $G$. Thus the obstruction map attached to $e'_v$ is obtained from $\kappa_G$ after identifying $F_0$ with $\Fil G$ through $c$ and $\mathcal I_X$ with $\Lie G$ through $g$. The primitive column $(-u_S,1)^t$ gives $g:\Lie G\simeq\mathcal I_X$. This gives
\begin{equation*}
    -u_S^{-1}\kappa_G\circ (\Fil(c)\otimes g^\vee): \Fil\mathcal A_0\otimes\mathcal I_X^{-1}\longrightarrow\omega_G.
\end{equation*}
On the canonical CM lift with action $c\iota_0c^{-1}$, the unit $c$ induces an integral isomorphism $\mathcal A_0\simeq G$. Hence $\Fil(c)$ is a unit and the fixed comparison has valuation $-1$. Equivalently, the quotient map $(r,s)\mapsto r-u_Ss$ sends this image generator to $-2u_S$ and the vector on the left of \eqref{ramified fixed comparison} to $2c$. The framing isometry that fixes $e'_v$ and acts by $-1$ on $W_v$ has determinant $-1$ modulo $u$. It exchanges the two Serre branches by \cite[Proposition~6.4]{RSZ1}. The comparison map is equivariant for this action, which fixes $e'_v$, so its valuation is the same on both branches. Their sum, divided by the factor $2$ in the normalization of the zero-dimensional Shimura variety, gives one copy of the ramification divisor. Its normalized degree is $\log d_{E/F}$, since all ramification is tame and quadratic. This finishes the computations at ramified places.

Now let $v\in\Sigma_f$, and let $\pi_v$ be a uniformizer of $F_v$. Along the CM cycle, the positive special homomorphism has unit norm and identifies $\mathcal{A}_0$ with the positive factor. Write the negative factor as $\mathcal A_-$, so the relative polarization on $\mathcal A_0\oplus\mathcal A_-$ has positive factor principal and negative factor $\pi_v\varepsilon$ times a principal polarization, for an integral unit $\varepsilon$. The two Kodaira--Spencer factors are
\begin{equation*}
    \Fil(\mathcal A_0)\otimes(\Lie\mathcal A_-)^{-1} =\mathcal L_1|_{\mathcal P},\qquad \Fil(\mathcal A_-)\otimes(\Lie\mathcal A_0)^{-1}.
\end{equation*}
Denote these two factors by $A$ and $B$. The polarization gives $j:B\to A$ with coefficient $\pi_v\varepsilon$ in integral bases, and horizontality gives $\kappa_-=\pm\kappa_+j$. The filtered tensor comparison above identifies Yuan--Zhang's determinant line bundle $N$ with $A\otimes B$. Thus
\begin{equation*}
    \chi=1\otimes j^{-1}:A^{\otimes2}\longrightarrow N, \qquad \mathrm{KS}_{\mathrm{YZ}}\circ\chi=\pm\kappa_+^{\otimes2}.
\end{equation*}
The map $\chi$ has valuation $-1$. The discriminant divisor in \cite[Theorem~4.10]{YZ1} has multiplicity $1$, so this square has valuation $1-1=0$. Thus the single-factor comparison is a unit on the CM cycle.

Now we determine where the comparison map vanishes or has poles on the other components of the special fibre. Let $D_v$ be the divisor of the polarization map
\begin{equation*}
    \pi_v\lambda^{-1}:\omega_{X^\vee}\longrightarrow\omega_X
\end{equation*}
on the chosen factor of the Drinfeld model. The two polarization maps compose to $\pi_v$; their zero divisors partition the components of the fibre. Note that the CM section avoids $D_v$, because $\pi_v\lambda^{-1}$ is a unit on its selected negative line.

Each geometric component is a projective line with $N_v+1$ nodes. On a component where $\lambda$ vanishes, write $a_0,a_1$ for the degrees of the two dual Lie lines. The two maps in \cite[Theorem~4.1]{Jun} have respectively $N_v+1$ and $N_v^2+1$ simple zeros. Hence
\begin{equation*}
    -a_0+a_1=N_v+1,\qquad -a_0+N_va_1=N_v^2+1, \qquad (a_0,a_1)=(-1,N_v).
\end{equation*}
The degrees are interchanged on the other type of component. Since the source filtration has degree zero on a geometric fibre and the dualizing line has degree $N_v-1$, the degrees on the two types are

\begin{equation*}
    \begin{array}{c|rr}
        &\lambda=0&\pi_v\lambda^{-1}=0\\
        \hline L_1&-1&N_v\\
        L_{B,E}&N_v-1&N_v-1\\
        \Delta&-(N_v+1)&N_v+1\\
        \mathcal O(D_v)&N_v+1&-(N_v+1).
    \end{array}
\end{equation*}
Thus $\Delta+D_v$ has degree zero on every component.

The same multiplicities can be read from the local equation $xy=\pi_v$. The two lattice maps have matrices $\operatorname{diag}(\pi_v,1)$ and $\operatorname{diag}(1,\pi_v)$, with filtration generators $(1,x)$ and $(y,1)$. Their coefficients are $x$ and $y$. Relative to the dualizing generator $dx/x=-dy/y$, the two Kodaira--Spencer maps have coefficients
\begin{equation*}
    x\,\frac{dx}{x},\qquad -y\,\frac{dx}{x}, \qquad \left(x\frac{dx}{x}\right)\left(-y\frac{dx}{x}\right) =-\pi_v\left(\frac{dx}{x}\right)^2.
\end{equation*}
Thus the local equations give the same component divisor.

At all remaining finite places, the unit-norm positive line is an integral orthogonal direct summand, and the wedge comparison is integral. Combining these facts with the ramified calculation gives

\begin{equation}\label{determinant line on CM cycle}
    \widehat{\deg}_E(\overline\Delta|_{\mathcal P}) =\LL_1\cdot\mathcal P+\log d_{E/F}.
\end{equation}

We now determine the contribution away from the CM section. On a connected generic component, the determinant character is evaluated on the rational arithmetic stabilizer preserving its lattice. It lies in the norm-one units of $E$ and is finite. Thus a power of the determinant line is generically trivial. After adding $\sum_{v\in\Sigma_f}D_v$, the divisor of a trivialization is vertical and has degree zero on every fibre component. The kernel of the fibre intersection matrix is generated by the whole fibre. At infinity its logarithmic norm is harmonic on a compact connected curve and hence constant. Therefore
\begin{equation*}
    \overline\Delta+\sum_{v\in\Sigma_f}(D_v,0)
\end{equation*}
is pulled back from the arithmetic base on each connected component. By \eqref{CM component weights} and $\mathcal P\cap D_v=\varnothing$, its $L_1$-normalized intersection is the value in \eqref{determinant line on CM cycle}.

Pass to a finite neat cover, such that the two types of components at $v\in\Sigma_f$ are separately defined. If one type has $m$ components, the other has $m$ as well, by counting edges in the $(N_v+1)$-regular bipartite dual graph. Components of the same type do not meet, and each component has self-intersection $-(N_v+1)$. Consequently,
\begin{equation*}
    D_v^2=-m(N_v+1),\qquad L_1D_v=mN_v, \qquad\deg L_1=m(N_v-1).
\end{equation*}
Since $\log N\nu=2\log N_v$ for the place $\nu$ of $E$ above $v$, normalization gives

\begin{equation*}
    \frac{(D_v,0)^2}{\deg L_1} =-\frac{2(N_v+1)}{N_v-1}\log N_v, \qquad \frac{\LL_1\cdot(D_v,0)}{\deg L_1} =\frac{2N_v}{N_v-1}\log N_v.
\end{equation*}

An arithmetic base class has square zero and has zero intersection with a vertical divisor. These formulas prove all three identities in \eqref{intersection of determinant difference line}.

Restricting $2\LL_1=\LL_{B,E}+\overline\Delta$ to $\mathcal P$, using \eqref{determinant line on CM cycle}, and converting the degree from $E$ to $F$ gives \eqref{exact point Hodge comparison}. Squaring the same identity yields
\begin{equation*}
    h_{\LL_1}(C_U) =2h_{\LL_B}^{\mathrm{Yuan}}(C_{U_B}) +\frac{\LL_1\cdot\overline\Delta}{\deg L_1} -\frac{\overline\Delta^2}{4\deg L_1}.
\end{equation*}
Here $\deg L_B=2\deg L_1$ and Yuan uses $h_{\LL_B}^{\mathrm{Yuan}}=\widehat{\deg}_F(\LL_B^2)/(2\deg L_B)$. The two vertical coefficients combine as
\begin{equation*}
    -\frac{2N_v}{N_v-1}+\frac{N_v+1}{2(N_v-1)} =-\frac{3N_v-1}{2(N_v-1)},
\end{equation*}
which proves \eqref{exact curve Hodge comparison}.
\end{proof}

The finite calculation in this Proposition is closely related to the Hodge calculation in \cite[Section~4.7]{Yuan1}. In both cases, after the relation on the generic fibre is fixed, the remaining problem is to measure how the corresponding Hodge lines differ on the integral model. This is reduced to explicit maps between rank-one filtration or Lie lines, and the local contribution is read from the valuation, or equivalently the cokernel length, of these maps. When the special fibre is reducible, the same calculation is carried out component by component and gives the corresponding vertical divisor.

\subsection{Proof of modular height formulas}
We apply the comparison just proved to the quaternionic height formulas, and prove two main theorem of this paper.

\subsubsection*{Modular heights of CM points}
For the rank-two datum, \cite[Theorem~1.7]{YZ1} gives
\begin{equation*}
    \widehat{\deg}_F(\LL_B|_{\mathcal P_{U_B}}) =-\frac{L_f'(0,\eta)}{L_f(0,\eta)} +\frac12\log\frac{d_\BB}{d_{E/F}}.
\end{equation*}
Substitution in \eqref{exact point Hodge comparison} gives

\begin{equation}\label{Height of CM cycle on curve}
    \LL_1\cdot\mathcal P =-2\frac{L_f'(0,\eta)}{L_f(0,\eta)}+\log d_\BB.
\end{equation}
This includes the inert almost self-dual places, with $d_\BB=\prod_{v\in\Sigma_f}N_v$.

For the standard datum, we use the auxiliary comparison in \eqref{aux-CM}. On its quaternionic curve the maximal-order CM formula of \cite[Theorem~1.7]{YZ1} gives
\begin{equation*}
    2\widehat{\deg}_F(\LL_B|_{\mathcal P_{U_B}}) =-2\frac{L_f'(0,\eta)}{L_f(0,\eta)} +\log d_\BB-\log d_{E/F}.
\end{equation*}
Since $\log d_\BB$ is supported in $S$, substitution in \eqref{aux-CM} yields

\begin{equation*}
    \LL\cdot\mathcal P+2\frac{L_f'(0,\eta)}{L_f(0,\eta)} \in\sum_{v\in S}\QQ\log N_v.
\end{equation*}

Now we can choose the standard local summands at every place above the rational primes underlying $S$ and repeat the construction. The same height belongs to a second such space whose rational-prime support is disjoint from the first. The logarithms of distinct rational primes are linearly independent over $\QQ$: clearing denominators in a relation and exponentiating reduces it to unique factorization. The intersection of the two spaces is therefore zero. We obtain
\begin{equation*}
    \LL\cdot\mathcal P=-2\frac{L_f'(0,\eta)}{L_f(0,\eta)}.
\end{equation*}

For a standard rank-two datum this also follows directly from $d_\BB=1$. This proves Theorem~\ref{Modular height of CM point}.

\subsubsection*{Modular heights of unitary Shimura curves}
In our notation, \cite[Theorem~1.1 and Section~5]{Yuan1} gives
\begin{equation*}
    \begin{aligned}
        h_{\LL_B}^{\mathrm{Yuan}}(C_{U_B}) ={}&\frac{\zeta_F'(2)}{\zeta_F(2)} +\sum_{v\in\Sigma_f}\frac{3N_v-1}{4(N_v-1)}\log N_v\\
        &-\Bigl(\gamma+\log(2\pi)-\frac12\Bigr)[F:\QQ] +\log|d_F|.
    \end{aligned}
\end{equation*}
Using \eqref{exact curve Hodge comparison} and \eqref{Height of CM cycle on curve}, this sum cancels with the component correction in \eqref{exact curve Hodge comparison}. Thus
\begin{equation*}
    \begin{aligned}
        h_{\LL_1}(C_U) ={}&-2\frac{L_f'(0,\eta)}{L_f(0,\eta)} +2\frac{\zeta_F'(2)}{\zeta_F(2)}+2\log|d_F|\\
        &-\bigl(2\gamma+2\log(2\pi)-1\bigr)[F:\QQ] +\log d_{E/F}+\log d_\BB.
    \end{aligned}
\end{equation*}
The functional equation of $\zeta_F$ rewrites this as
\begin{equation*}
    h_{\LL_1}(C_U) =-2\frac{L_f'(0,\eta)}{L_f(0,\eta)} -2\frac{\zeta_F'(-1)}{\zeta_F(-1)}-[F:\QQ] +\log d_{E/F}+\log d_\BB.
\end{equation*}
Propositions~\ref{relation of lattice and order} and~\ref{relation of lattice and order nonsplit case} identify $\Sigma_f$ with the finite nonsplit places of the associated quaternion algebra. This proves Theorem~\ref{Modular height of curve}.

\section{Height series}\label{Height series}

In this section we study the \textit{height series}. The height series comes from the arithmetic intersection of an extension of the generating series to the integral model and a degree-zero arithmetic 1-cycle. In order to obtain our final result, we also need to do some explicit computation of these arithmetic intersections.

The arrangement of this section is as follows: we first review the relevant arithmetic intersection theory, focusing on normalized admissible extensions. Next, we define the generating series of divisors and its arithmetic version. Finally, we define the height series using arithmetic intersections. Throughout this section and Section~\ref{Explicit local intersection}, we use the standard datum of Assumption~\ref{Assumption}.

\subsection{Arithmetic Chow cycles and admissible extension}\label{Arakelov}
We fix the arithmetic intersection convention and the normalized admissible representatives used in the height series. At finite places we first identify these representatives on the integral model, and then compute their intersections.

\subsubsection*{Terminology for arithmetic intersection theory}
Let $\mathcal X$ be a regular proper arithmetic variety of relative dimension $n$ over $\mathcal O_E$, with generic fibre $X$. We use the convention

\begin{equation*}
    dd^c=\frac{i}{\pi}\partial\bar\partial.
\end{equation*}

An arithmetic cycle is a pair $(Z,g_Z)$, where $g_Z$ is a Green current satisfying $dd^cg_Z+\delta_Z=\omega_Z$ for a smooth form $\omega_Z$. For a rational function $f$ on an integral subscheme $W$ with inclusion $i_W$, the principal arithmetic cycle is

\begin{equation*}
    \widehat{\operatorname{div}}_W(f) =\bigl(\operatorname{div}_W(f),-i_{W,*}\log|f|\bigr).
\end{equation*}

The quotient by arithmetic rational equivalence is the arithmetic Chow group. In particular, a Hermitian line bundle $\overline L=(L,\|\cdot\|)$ defines

\begin{equation*}
    \widehat c_1(\overline L) =\bigl(\operatorname{div}(s),-\log\|s\|\bigr)
\end{equation*}
for a nonzero rational section $s$. The arithmetic degree sums over all complex embeddings:
\begin{equation*}
    \widehat{\deg}\left(\sum_{\nu\nmid\infty}a_\nu[\nu] +\sum_{\sigma:E\hookrightarrow\CC}a_\sigma[\sigma]\right) =\sum_{\nu\nmid\infty}a_\nu\log N\nu +\sum_{\sigma:E\hookrightarrow\CC}a_\sigma.
\end{equation*}
This agrees with the Gillet--Soul\'e convention after pairing conjugate embeddings; see \cite{GS}.

Let $D$ be a divisor and $P$ a zero-cycle on $X$. For representatives whose generic supports are disjoint and whose finite intersections have finite length, the divisor--cycle pairing is
\begin{equation*}
    \widehat D\cdot P =\sum_{\nu\nmid\infty}(D,P)_\nu\log N\nu +\sum_{\sigma:E\hookrightarrow\CC}g_{D,\sigma}(P_\sigma).
\end{equation*}
For a horizontal closure $D$ and a vertical correction $V$, the traditional decomposition is
\begin{equation*}
    \widehat D\cdot P=i(D)+j(D),\qquad i(D)=D\cdot P+\sum_\sigma g_{D,\sigma}(P_\sigma),\quad j(D)=V\cdot P.
\end{equation*}
A derived intersection computes the intersection of the chosen finite representatives.

On a local RSZ realization, assume that the derived tensor product has bounded cohomology with finite-length cohomology groups over $\mathcal O_{\wt E,\wt\nu}$. Define
\begin{equation}\label{derived finite intersection}
    \operatorname{Int}_{\wt\nu}(D,P) :=\chi\left(\widehat{\mathcal M}^{\mathrm{loc}}_{\wt\nu},
    \mathcal O_{\wt{\mathcal D}_{\wt\nu}} \overset{\mathbf L}{\otimes}_{\mathcal O_{\widehat{\mathcal M}^{\mathrm{loc}}_{\wt\nu}}} \mathcal O_{\wt{\mathcal P}_{\wt\nu}}\right),
\end{equation}
where
\begin{equation*}
    \chi(\widehat{\mathcal M}^{\mathrm{loc}}_{\wt\nu},\mathcal F)
    =\sum_{i,j}(-1)^{i+j} \operatorname{length}_{\mathcal O_{\wt E,\wt\nu}} H^j(\widehat{\mathcal M}^{\mathrm{loc}}_{\wt\nu},\mathcal H^i(\mathcal F)).
\end{equation*}
This is the Euler--Poincar\'e characteristic of \cite[(9.7)]{MZ}. On a regular model it agrees with Serre's intersection multiplicity. The normalized pairing is
\begin{equation}\label{DP term}
    \widehat D\cdot\mathcal P =\frac{1}{\delta_Z[\wt E:E]} \left(\sum_{\wt\nu\nmid\infty}\operatorname{Int}_{\wt\nu}(D,P)\log N\wt\nu
    +\sum_{\wt\sigma:\wt E\hookrightarrow\CC} g_{D,\wt\sigma}(P_{\wt\sigma})\right).
\end{equation}
The projection formula makes this independent of a finite extension of the common reflex field. For a branch containing the CM cycle, the arithmetic pairing is the degree of the restricted metrized divisor line. We compute this restriction in Lemma~\ref{Vanish of improper intersection}.

\subsubsection*{Admissible extension}
Fix the metrized Hodge line bundle $\LL$, and put
\begin{equation*}
    \xi=\frac{L^n}{\deg_L(X)},\qquad \widehat\xi=\frac{\LL^n}{\deg_L(X)}.
\end{equation*}
We use the normalized admissible extension of \cite[Appendix~A.2--A.3]{Qiu}, with the convention above. For $\widehat D=(D+V,\mathcal G_D)$, the finite admissibility condition says that the intersections of $(D+V)\cdot L^{n-1}$ with vertical divisors are proportional to those of $L^n$, with the generic degree ratio on each connected component. The vertical correction is normalized by $(V\cdot L^n)_\nu=0$. On a connected complex component, the archimedean conditions are
\begin{equation*}
    (\delta_D+dd^c\mathcal G_D)\wedge c_1(\overline L)^{n-1} =\frac{\deg_L D}{\deg_L X}c_1(\overline L)^n, \qquad \int_X\mathcal G_D\,c_1(\overline L)^n=0.
\end{equation*}
Here both degrees refer to that component. Omitting the last equality gives a weakly admissible Green function.

For the standard maximal lattice data, the integral model is smooth at every finite place by \cite[Lemma~4.4.6]{Qiu}. Hence, by \cite[Lemma~A.2.4]{Qiu}, the normalized finite admissible representative is simply the Zariski closure of the generic divisor. At a nonsplit place, the local uniformization of \cite[Proposition~6.1.5 and Section~6.1.2]{Qiu} identifies this closure with the corresponding weighted Kudla--Rapoport lifting spaces, which are flat Cartier divisors and are the closures of their generic fibres. At a split place, the same flatness statement follows from the EL calculation in Proposition~\ref{Ordinary case}. Therefore the proper finite intersections of these divisors are computed by \eqref{derived finite intersection}.

We also need this description to be compatible with changes of level. If a divisor is pulled back from the chosen maximal level, \cite[Proposition~A.3.8]{Qiu} identifies its admissible extension with the admissible pullback, while \eqref{CM-level-pairing} shows that its pairing with the fixed CM cycle is unchanged after passing between the two levels. If the divisor is defined only at a smaller level, \cite[Corollary~A.4.1]{Qiu} compares its admissible extension with its Zariski closure after pairing with the full pullback of a cycle from the smooth maximal-level model. Here the relevant CM cycle is precisely the weighted sum of translates defined in Section~\ref{Hodge bundle and CM cycle}.

Admissible extensions of a fixed generic Chow class differ by arithmetic classes of connected components; see \cite[Lemma~A.3.6]{Qiu}. Pairing such a difference with a cycle of degree zero on every component remains zero. For the normalized Hodge-cycle class, we simply write
\begin{equation*}
    \widehat D^{\mathrm{adm}}\cdot\widehat\xi =\frac{\widehat D^{\mathrm{adm}}\cdot\LL^n}{\deg_L(X)}.
\end{equation*}

\subsubsection*{Our setting}
Our height series uses the difference
\begin{equation*}
    \widehat D^{\mathrm{adm}}\cdot(\mathcal P-\widehat\xi).
\end{equation*}
Its generic zero-cycle has degree zero on every component by \eqref{CM component weights}. It therefore depends only on the generic divisor class. When the cycles are defined after a finite extension, all degrees are normalized as in \eqref{DP term}.

\begin{remark}
On a connected complex component $X_i$, let $g_D$ be weakly admissible. Then its normalized Green function is
\begin{equation*}
    \mathcal G_D=g_D-\frac{1}{\deg_L X_i}\int_{X_i}g_D\,c_1(L)^n.
\end{equation*}
By \eqref{CM component weights}, $\mathcal P$ and $\widehat\xi$ give the same weight to $X_i$. This constant shift changes $\widehat D\cdot\mathcal P$ and $\widehat D\cdot\widehat\xi$ by equal amounts and cancels in their difference.
\end{remark}

\subsection{Generating series of special divisors}
In this subsection, based on the discussion in \cite{Liu1}, we introduce the generating series of special divisors associated with unitary Shimura varieties.

We also use the moduli spaces parametrizing RSZ special divisors. Their normalized pushforwards agree with the unitary divisors under \eqref{unitary and RSZ}. Their local lifting branches provide the moduli interpretation needed for the intersection calculations.

\subsubsection*{Special divisors}

We recall the Kudla divisors on the generic unitary Shimura variety.  Let $y\in V$ satisfy $q(y)>0$, put $W_y=Ey$, and let $g\in G(\Af)$.  The corresponding divisor is
\begin{equation*}
    Z(y,g)_U :=G_{W_y}(\QQ)\backslash D_{W_y}\times G_{W_y}(\Af)gU/U \longrightarrow X_U,
\end{equation*}
where
\begin{equation*}
    D_{W_y}=\{z\in D\mid \langle z,y\rangle=0\}
\end{equation*}
and $G_{W_y}$ is the pointwise stabilizer of $W_y$.  This notation includes the usual sum over the relevant finite double cosets when the displayed map is not generically injective.

The divisor depends only on the $U$-orbit of the finite adelic vector
\begin{equation*}
    x=g^{-1}y\in V_f.
\end{equation*}
Accordingly, we call $x\in V_f$ \emph{admissible} if it is of this form for some $y\in V$ with $q(y)>0$ and some $g\in G(\Af)$; for such $x$ we write $Z(x)_U=Z(y,g)_U$, and for a non-admissible $x$ we put $Z(x)_U=0$.  For $u\in U$ and $a\in E^\times$,
\begin{equation}\label{invariance property}
    Z(x)_U=Z(uxa)_U.
\end{equation}
We also write $D_x=D_{W_y}$ on the corresponding complex component.

The orthogonal complement $W_y^\perp$ has dimension $n$ and defines the lower-dimensional unitary Shimura datum.  Hence, as a cycle with its natural finite map,
\begin{equation*}
    Z(y,g)_U\cong X_{W_y^\perp,U_{W_y^\perp,g}},
\end{equation*}
where the level is obtained by intersecting $gUg^{-1}$ with $\U(W_y^\perp)(\Af)$.  This is a statement on the generic fiber; its integral realization is the corresponding Kudla--Rapoport divisor on the RSZ model.

\begin{remark}
We will discuss the properties of $\Lambda\cap V_x^\perp$ in detail in the third paper of this series. In fact, these properties are crucial for our computations there.
\end{remark}

When $n=1$, a degree-one CM cycle is a special divisor on the unitary Shimura curve.   For $\Phi=\Phi_f\in\mathcal S(V_f)^U$ and $t\in F_+$, define the weighted divisor
\begin{equation*}
    Z_t(\Phi)_U :=\sum_{\substack{x\in U\backslash V_f\\q(x)=t}} \Phi_f(x)Z(x)_U,
\end{equation*}
where non-admissible terms are zero. For $a\in E^\times$, let $a_f$ be its finite adelic component. The Weil action of \cite[Section~2.1]{Guo1} is
\begin{equation*}
    r(m(a_f))\Phi(x) =\chi_{\BV,f}(a_f)|a_f|_{\BA_{E,f}}^{(n+1)/2}\Phi(a_fx).
\end{equation*}
Changing variables from $x$ to $ax$ and using \eqref{invariance property} gives
\begin{equation}\label{invariance property of weighted divisor}
    Z_t(r(m(a_f))\Phi)_U =\chi_{\BV,f}(a_f)|a_f|_{\BA_{E,f}}^{(n+1)/2} Z_{\Nm(a)t}(\Phi)_U.
\end{equation}

\subsubsection*{Special divisors on RSZ Shimura varieties}

The RSZ formulation gives the moduli-theoretic realization of the preceding divisors.  Let
\begin{equation*}
    (A_0,\psi_0,\lambda_0,\bar\eta_0; A,\psi,\lambda,\bar\eta)
\end{equation*}
be a point of the generic RSZ moduli stack over a connected locally noetherian scheme over $\wt E$.  The $E$-vector space $\Hom_E^0(A_0,A)$ carries the Hermitian form
\begin{equation*}
    h'(x,y)=\lambda_0^{-1}\circ y^\vee\circ\lambda\circ x \in\End_E^0(A_0)\cong E.
\end{equation*}
For $t\in F_+$, the generic Kudla--Rapoport divisor $\wt Z_t$ parametrizes such tuples together with $x\in\Hom_E^0(A_0,A)$ satisfying $h'(x,x)=t$ and the lattice condition determined by the level structure. Under a level trivialization this condition is expressed by
\begin{equation*}
    \Hom_E^0(A_0,A) \hookrightarrow \Hom_{\BA_{E,f}}(\widehat V(A_0),\widehat V(A)) \xrightarrow{\ \eta\ } \wt V\otimes_E\BA_{E,f},
\end{equation*}
and requires the image of $x$ to lie in $\wt\Lambda_f$ for the standard lattice datum. The divisor $\wt Z(x_0)$ for a fixed admissible $x_0\in\wt V_f$ is defined similarly.  These generic divisors are finite and unramified over the RSZ Shimura variety and have codimension one.

In the isogeny category, the endomorphism
\begin{equation*}
    x\,h'(x,x)^{-1}\lambda_0^{-1}x^\vee\lambda\in\End_E^0(A)
\end{equation*}
is an idempotent. Its complementary factor in $A$, with the induced action, polarization, and level, realizes the Hermitian space $V_x^\perp$. Thus
\begin{equation*}
    \wt Z(x)\cong S(\wt G_x,\{h_{\wt G_x}\})_{K_{\wt G_x}}.
\end{equation*}
On the integral RSZ model, the special divisor is the lifting locus of the special quasi-homomorphism attached to $x$.  On a neighbourhood of the CM branch in the relative Rapoport--Zink space, it is a Cartier divisor or the empty locus.  Under \eqref{unitary and RSZ}, $\wt Z_t$ is the pullback of $Z_t$.

\subsubsection*{Generating series}
For $\tau=(\tau_k)\in\mathcal H^d$, put $q^t=\exp(2\pi i\sum_k t_k\tau_k)$. The generating series in the generic Chow group is
\begin{equation*}
    Z_\Phi(\tau)=[L^\vee]\Phi(0)+\sum_{t\in F_+}Z_t(\Phi)q^t, \qquad Z_{\Phi,*}(\tau)=\sum_{t\in F_+}Z_t(\Phi)q^t.
\end{equation*}
The level maps $f:X_{U'}\to X_U$, for $U'\subset U$, satisfy
\begin{equation*}
    f^*Z_t(\Phi)_U=Z_t(\Phi)_{U'},\qquad f^*L_U=L_{U'}.
\end{equation*}
Thus the series is compatible in the direct limit under pullback; see \cite{Liu1,YZZ1}.

\begin{theorem}[{\cite[Theorem~3.5]{Liu1}}]
The generating series $Z_\Phi(\tau)$ is a modular form of weight $\mathfrak m$, with values in $\operatorname{Ch}^1(X_U)_\CC$.
\end{theorem}
For the standard archimedean Gaussian, its adelic form is
\begin{equation*}
    Z(g,\Phi)=[L^\vee]r(g)\Phi(0) +\sum_{t\in F_+}\sum_{\substack{y\in U\backslash V_f\\q(y)=t}} r(g)\Phi(y)Z(y)_U.
\end{equation*}
Here $g\in\UU(\BA_F)$ and the archimedean factor of each positive coefficient is $W_t^{\mathfrak m}(g_\infty)$. The constant term is $[L^\vee]r(g)\Phi(0)$.

The degree of the generating series is $\deg_L Z(g,\Phi)=Z(g,\Phi)\cdot c_1(L)^{n-1}$. The geometric Siegel--Weil formula is
\begin{equation}\label{Geometric Siegel Weil}
    \deg_L Z(g,\Phi)=-\deg_L(X_U)E(0,g,\Phi);
\end{equation}
see \cite[Remark~4.2.5]{Qiu} and \cite{Kud1}. We use $E_*$ for its positive Fourier part. On the generic curve $c_1(L_B)=2c_1(L)$, as in Remark~\ref{Kodaira-Spencer remark}; this accounts for the factor $1/2$ in the quaternionic convention of \cite[Proposition~4.2]{YZZ2}.

\subsection{Arithmetic generating series}

We now combine the admissible arithmetic class at infinity with the moduli-theoretic RSZ divisor at finite places.  The resulting object is an adelic arithmetic divisor class normalized by the toric and reflex-field degrees of Section~\ref{Integral model}.

\subsubsection*{The admissible arithmetic class}
We lift the generic divisor series using the normalized admissible extensions of Section~\ref{Arakelov}. Define
\begin{equation*}
    \wh{\mathcal Z}_\Phi(\tau) :=[\LL^\vee]\Phi(0) +\sum_{t\in F_+}\wh{\mathcal Z}_t q^t.
\end{equation*}
Here $\wh{\mathcal Z}_t$ denotes the following pair of compatible data:
\begin{enumerate}[(1)]
\item its archimedean component is the normalized $\LL$-admissible Green function for $Z_t$;
\item at a standard maximal level, its finite representative is the closure identified with the weighted Kudla--Rapoport branches in Section~\ref{Arakelov}.
\end{enumerate}
The corresponding class is unique modulo arithmetic classes pulled back from the base.  We also put
\begin{equation*}
    \wh{\mathcal Z}_{\Phi,*}(\tau) :=\sum_{t\in F_+}\wh{\mathcal Z}_t q^t
\end{equation*}
and use $\wh{\mathcal Z}(g,\Phi)$ for the adelic notation.

The modularity of the geometric generating series gives the following modularity statement after pairing with a degree-zero arithmetic cycle.

\begin{lemma}\label{Almost modularity}
Let $P_0$ be a zero-cycle of degree zero on every connected component of $X_U$, and let $\mathcal P_0$ be its closure. The series
\begin{equation*}
    \widehat{\mathcal Z}(g,\Phi)\cdot\mathcal P_0
\end{equation*}
is a holomorphic automorphic form of weight $\mathfrak m$.
\end{lemma}

\begin{proof}
The pairing of an admissible extension with $\mathcal P_0$ is a linear functional on $\operatorname{Ch}^1(X_U)$. Indeed, two lifts of the same class differ by componentwise arithmetic base classes, which pair with $\mathcal P_0$ by its component degrees and therefore give zero. Applying this functional to geometric modularity, in the form of \cite[Lemma~4.3.6]{Qiu}, proves the assertion.
\end{proof}

\begin{definition}
The height series is
\begin{equation}\label{height series}
    \left(\widehat{\mathcal Z}_*(g,\Phi) -\frac{\deg_L Z_*(g,\Phi)}{\deg_L X_U}\LL\right) \cdot(\mathcal P-\widehat\xi).
\end{equation}
\end{definition}

To see its modularity and its constant terms, first use the full generating series in the bracket:
\begin{equation}\label{full height bracket}
    \left(\widehat{\mathcal Z}(g,\Phi) -\frac{\deg_L Z(g,\Phi)}{\deg_L X_U}\LL\right) \cdot(\mathcal P-\widehat\xi).
\end{equation}
As in Lemma~\ref{Almost modularity}, the degree-zero pairing defines a linear functional on the generic Chow group. The degree map and the geometric Siegel--Weil formula are adelic, so the constant term at every cusp is computed after the corresponding translate of $\Phi$. In the bracket this term is
\begin{equation*}
    -r(g)\Phi(0)\LL -\frac{-\deg_L X_U\,r(g)\Phi(0)}{\deg_L X_U}\LL=0.
\end{equation*}
Thus \eqref{full height bracket} is cuspidal and equals \eqref{height series}.

For later comparison, we write the two arithmetic pairings separately:
\begin{equation*}
    \begin{aligned}
        &\left(\widehat{\mathcal Z}_*(g,\Phi) -\frac{\deg_L Z_*(g,\Phi)}{\deg_L X_U}\LL\right)\cdot(\mathcal P-\widehat\xi)\\
        ={}&\widehat{\mathcal Z}_*(g,\Phi)\cdot\mathcal P -\widehat{\mathcal Z}_*(g,\Phi)\cdot\widehat\xi\\
        &-\frac{\deg_L Z_*(g,\Phi)}{\deg_L X_U} \LL\cdot(\mathcal P-\widehat\xi).
    \end{aligned}
\end{equation*}

\subsubsection*{Integral Kudla--Rapoport divisors on the RSZ model}
Let $t\in\mathcal O_F^+$.  On the integral RSZ moduli stack, define $\wt{\mathcal Z}_t$ to parametrize tuples
\begin{equation*}
    (A_0,\psi_0,\lambda_0,A,\psi,\lambda,\text{level data};x)
\end{equation*}
with the RSZ Eisenstein conditions and with an $\mathcal O_E$-linear homomorphism
\begin{equation*}
    x:A_0\longrightarrow A, \qquad h'(x,x)=t.
\end{equation*}

The special divisor has a finite forgetful morphism
\begin{equation*}
    \wt{\mathcal Z}_t\longrightarrow\mathcal M_{K_{\wt G}}(\wt G).
\end{equation*}
Its generic fibre is $\wt Z_t$. On a relative Rapoport--Zink chart, a fixed special homomorphism gives its complete local Cartier lifting divisor, or the empty locus.

\section{Explicit local intersections}\label{Explicit local intersection}

In this section we do some explicit computations of the arithmetic intersection $\wh{\mathcal Z}_{*}(g,\Phi)\cdot\mathcal P$. We compute separately at each place. Combined with the computation of the derivative series in \cite[Section~4]{Guo1}, these calculations prove the main Theorem~\ref{main theorem of arithmetic Siegel-Weil I} of this paper.

We first do the finite local computations at $g_v=1$, using the Schwartz functions fixed in \cite[Section~4.1]{Guo1}. The formulas for general $g_v$ are then obtained from the Weil action. The compact and upper-triangular cases needed below are recalled in Section~\ref{Proof of the main theorem}.

\subsection{The improper intersections}
We first consider the special divisors that contain a CM point. We compare their restrictions to the CM cycle with the inverse Hodge bundle. This comparison gives the logarithmic term that cancels the corresponding term in the analytic series in \cite{Guo1}.

\begin{lemma}\label{Vanish of improper intersection}
Let $x\ne0$, and let $\mathcal Z(x)$ be the local divisor defined by the condition that $x$ lifts to a homomorphism. It contains the chosen generic CM point if and only if $x\in W^\perp$. In this case, let $C_x$ be the closure of the unique reduced generic branch through that point, and write
\begin{equation*}
    \mathcal Z(x)=C_x+R_x.
\end{equation*}
The obstruction to lifting $x$ gives
\begin{equation}\label{special conormal comparison}
    \mathcal O(-\mathcal Z(x))|_{\mathcal Z(x)} \cong\mathcal L|_{\mathcal Z(x)},\qquad \mathcal O(\mathcal Z(x))|_{\mathcal Z(x)} \cong\mathcal L^{-1}|_{\mathcal Z(x)}.
\end{equation}
In the almost modular case, these isomorphisms are taken on the open subset where the image in \eqref{almost modular Hodge extension} is a line subbundle of $\Lie\mathcal A$; this open subset contains the CM cycle.

Equip $\mathcal O(C_x)|_{\mathcal P}$ with the regularized metric given by $Q_0$. Let $(R_x,\mathcal P)_{\mathrm{fin}}$ denote the normalized finite contribution of $R_x$ in \eqref{DP term}. After averaging over the degree-one CM cycle, the expression
\begin{equation}\label{containing intersection}
    \widehat{\deg}\bigl(\overline{\mathcal O(C_x)}|_{\mathcal P}\bigr) +(R_x,\mathcal P)_{\mathrm{fin}}+\log|q(x)|_f
\end{equation}
is independent of $x\in W^\perp\setminus\{0\}$.
\end{lemma}

\begin{proof}
For convenience, we call $\mathcal Z(x)$ with $x\in W^\perp$ the containing divisor. We compare $\mathcal O(C_x)|_{\mathcal P}$ with $\mathcal L^{-1}|_{\mathcal P}$. At finite places, the zeros of this comparison give exactly the intersections with $R_x$. At infinity, we compute its norm from the logarithmic singularity of $Q_0$. The product formula then gives \eqref{containing intersection}.

\smallskip
\noindent\textbf{Step 1. } We first identify the line bundle associated with the containing divisor. The lifting obstruction gives the comparison with $\mathcal L^{-1}$, and after separating $\mathcal Z(x)=C_x+R_x$, its finite discrepancy along $\mathcal P$ is measured exactly by the intersections with $R_x$.

The negative CM line $W$ lies in the hyperplane $\langle z,x\rangle=0$ if and only if $x\in W^\perp$. We work on an \'etale chart with a fixed special homomorphism $x$ and a chosen CM point. The derivative of the hyperplane equation in the direction of $x$ is a nonzero multiple of $q(x)$. Thus its generic branch through the CM point is unique and reduced. Since the ambient local ring is regular and hence factorial, its closure $C_x$ is Cartier. We choose local equations
\begin{equation*}
    \mathcal Z(x)=(F_x),\qquad C_x=(f_x),\qquad F_x=f_xG_x,\qquad R_x=\operatorname{div}(G_x).
\end{equation*}
The restriction of $G_x$ to the generic CM point is nonzero.

We use Grothendieck--Messing theory \cite{Mes} to describe the equation $F_x$. The obstruction to lifting $x$ is the map from $\Fil\mathcal A_0$ to the Lie quotient defining \eqref{line bundle of modular forms}; see \cite[Definition~7.1 and Proposition~7.8]{LMZ}. In the ramified case, recall the uniformizer $u$ and its image $u_S$ on the base from Section~\ref{Hodge bundle and CM cycle}. For the Dieudonn\'e crystal $\mathbb D(\mathcal A_0)$, we have
\begin{equation*}
    \Fil\mathcal A_0=(\iota_{\mathcal A_0}(u)-u_S)\mathbb D(\mathcal A_0).
\end{equation*}
The crystalline map commutes with the $\mathcal O_E$-action, so its image in $\Lie\mathcal A$ lies in
\begin{equation*}
    \operatorname{im}\bigl(\iota_{\mathcal A}(u)-u_S\bigr) =\mathcal I_{\mathcal A}.
\end{equation*}
Proposition~\ref{restriction of Hodge bundle on curve} shows that along the CM cycle the image $I_{\mathcal A}$ is an invertible direct summand of $\Lie\mathcal A$. Thus in each case the obstruction takes values in $\mathcal L^{-1}$. Its coefficient generates the ideal of $\mathcal Z(x)$ modulo its square, and Nakayama's lemma gives the local equation $F_x$. This proves \eqref{special conormal comparison}.

We now restrict to $C_x$. The factorization $F_x=f_xG_x$ gives the map
\begin{equation*}
    \mathcal O(C_x)|_{\mathcal P} \longrightarrow\mathcal O(\mathcal Z(x))|_{\mathcal P} \cong\mathcal L^{-1}|_{\mathcal P}, \qquad f_x^{-1}\longmapsto G_xF_x^{-1}.
\end{equation*}
At a finite place, we compute on each normalized local component of $\mathcal P$. The order of this map is the order of $G_x|_{\mathcal P}$. The Cartier resolution of $R_x$ gives the same multiplicity:
\begin{equation*}
    \operatorname{length}\bigl(\mathcal O_{\mathcal P}/(G_x|_{\mathcal P})\bigr).
\end{equation*}
Consequently, the normalized finite degree of the comparison is $(R_x,\mathcal P)_{\mathrm{fin}}$.

\smallskip
\noindent\textbf{Step 2. } We next determine the metric in the preceding comparison. On the complex fibre the same obstruction is given by the pairing with $x$, so the logarithmic singularity of $Q_0$ gives the norm of $\mathcal O(C_x)|_{\mathcal P}$. Combining this with the finite calculation by the product formula yields the dependence on $q(x)$.

Let $s$ be a local nonvanishing section of the tautological line. We use the polarizations $\lambda_0$ and $\lambda$ to identify the obstruction with the pairing
\begin{equation*}
    s\longmapsto\lambda_0^{-1}x^\vee\lambda(s)=\langle s,x\rangle.
\end{equation*}
Hence the map in Step 1 is also the map given by this pairing on the complex fibre.

We write
\begin{equation*}
    \langle s,x\rangle=b_xf_x.
\end{equation*}
The argument of $Q_0$ is
\begin{equation*}
    1+\frac{|b_xf_x|^2}{q(x)\|s\|^2}.
\end{equation*}
By \eqref{Green constant at divisor}, we obtain the regularized value by adding $\log|f_x|$ and restricting to $C_x$. Up to the constant depending only on $n$ in that formula, the result is
\begin{equation*}
    \frac12\log q(x)+\log\|s\|-\log|b_x|.
\end{equation*}
At finite places this map is given by $G_x$ up to an integral unit, and at infinity it is given by $b_x$. We apply the product formula. The restricted degree is $-\log|q(x)|_f-(R_x,\mathcal P)_{\mathrm{fin}}$ plus a term determined by the fixed metrized Hodge bundle and the constant in \eqref{Green constant at divisor}. This gives the relation in \eqref{containing intersection} on each CM component.

\smallskip
\noindent\textbf{Step 3. } Finally, we check that the preceding calculation is compatible with scalar multiplication and with the weighted CM cycle. This allows us to pass from the calculation for a fixed special homomorphism and a fixed CM point to the weighted sum appearing in the height series.

Now we record the effect of replacing $x$ by a scalar multiple, since we will use it for the  Weil action. Write $\operatorname{ob}_x$ for the obstruction map under the polarization identification above. We have
\begin{equation}\label{obstruction scaling}
    \operatorname{ob}_{\alpha x}=\bar\alpha_S\operatorname{ob}_x, \qquad \alpha\in\mathcal O_{E_v},
\end{equation}
where $\bar\alpha_S$ is the image of $\bar\alpha$ on the base. At a split place, $\operatorname{ord}_\nu(\bar\alpha)=\operatorname{ord}_{\bar\nu}(\alpha)$. We sum over conjugate places and obtain, for $\alpha\in E^\times$,
\begin{equation}\label{conjugate-valuations}
    \sum_{\nu\nmid\infty}\operatorname{ord}_\nu(\bar\alpha)\log N\nu =\sum_{\nu\nmid\infty}\operatorname{ord}_\nu(\alpha)\log N\nu.
\end{equation}

We apply the preceding calculation to each special homomorphism and each CM point, with the weights in the divisor series and in \eqref{CM component weights}. Automorphisms that preserve the homomorphism and the CM point also preserve $C_x$. Thus the decomposition $\mathcal Z(x)=C_x+R_x$ and the comparison of finite degrees pass to the weighted sum. Since $\mathrm U(W)$ acts trivially on $W^\perp$, the vector $x\in W^\perp$ and hence its Schwartz weight are unchanged along the CM orbit. Thus the same Schwartz weight multiplies every term in the CM average. Moreover, at a ramified place, we compute the contribution on each of the two Serre branches separately and then combine them with the normalization of the CM cycle in \eqref{CM component weights}.

The divisors indexed by $y\notin W^\perp$ give the proper local intersections. For $y\in W^\perp\setminus\{0\}$, we include both the restricted line bundle of $C_y$ and the finite intersections with $R_y$. The other generic branches contribute their usual Green values. Finally, we apply Proposition~\ref{extra term} once to the full weighted divisor to normalize its Green function to have mean zero.
\end{proof}

The term $-\log|q(x)|_f$ obtained above cancels the term $+\log|q(x)|_f$ in the projected analytic series. The remaining contribution is a constant multiple of $\Phi|_{W^\perp}$. We combine it with the proper local intersections in Section~\ref{Proof of the main theorem}.

\subsection{Local intersections at archimedean places}
We compute the weakly admissible Green function, and its comparison with the archimedean kernel of \cite[Theorem~3.3]{Guo1}. The computations are coefficientwise for positive Fourier indices. At each archimedean place $v$ we use the nearby coherent space $V(v)$ and its complex uniformization.

\subsubsection*{Standard coordinate}
Fix the complex Hermitian space of signature $(n,1)$ at $v$. Choose an orthogonal complex basis with norms $1,\ldots,1,-1$, whose last vector spans the CM line. In the affine chart with last coordinate $1$, the symmetric domain is
\begin{equation*}
    D=\{z=(z_1,\ldots,z_n)\in\CC^n:\textstyle\sum_i|z_i|^2<1\}.
\end{equation*}
Write $h(z)=1-\sum_i|z_i|^2$. The tautological section $s_z=(z_1,\ldots,z_n,1)$ has squared norm $h(z)$, and
\begin{equation*}
    \omega=c_1(\overline L_D) =-\frac{i}{2\pi}\partial\bar\partial\log h.
\end{equation*}
The CM point is $z_0=0$. For a positive vector $x=(b_1,\ldots,b_{n+1})$, its hyperplane is
\begin{equation*}
    D_x:\quad \sum_{i=1}^n z_i\overline{b_i}-\overline{b_{n+1}}=0.
\end{equation*}
These coordinates are chosen separately at each complex place.

For $q(x)>0$ put
\begin{equation*}
    R_x(z)=-\frac{|\langle x,s_z\rangle|^2}{\langle s_z,s_z\rangle}, \qquad \rho_x(z)=\frac{R_x(z)}{q(x)},\qquad t(x,z)=1+\rho_x(z).
\end{equation*}
The function $\rho_x$ is unchanged when $x$ is multiplied by a nonzero scalar and satisfies $\rho_{gx}(gz)=\rho_x(z)$. At the CM point, for $x=x_1+x_2$ in $W^\perp\oplus W$, it is
\begin{equation*}
    \rho_x(z_0)=-\frac{q(x_2)}{q(x)}.
\end{equation*}
For $n=1$, the hyperbolic cosine $d$ used in \cite[Section~4.3]{Yuan1} is related to this variable by
\begin{equation*}
    t=\frac{d+1}{2}.
\end{equation*}

\subsubsection*{Weakly admissible Green function}
We construct the radial kernel for a positive divisor vector and then sum it over the arithmetic quotient. By unitary invariance, it loss no generality to take $x=(1,0,\ldots,0)$, so that $t=1+|z_1|^2/h$. At $(r,0,\ldots,0)$ the K\"ahler matrix is $\operatorname{diag}(h^{-2},h^{-1},\ldots,h^{-1})$. Direct differentiation gives
\begin{equation*}
    |\partial t|^2=t(t-1),\qquad \operatorname{tr}(\partial\bar\partial t)=(n+1)t-n.
\end{equation*}

Taking the trace with respect to $\omega$ gives
\begin{equation}\label{radial trace identity}
    \left(\frac{i}{2\pi}\partial\bar\partial Q(t)\right)\wedge\omega^{n-1} =\frac{t(t-1)Q''(t)+((n+1)t-n)Q'(t)}{n}\,\omega^n.
\end{equation}
Thus the eigenvalue $s(s+n)/n$ corresponds to
\begin{equation*}
    (t-t^2)Q''+(n-(n+1)t)Q'+s(s+n)Q=0.
\end{equation*}

The solution which decays at infinity and has leading singularity $-\tfrac12\log(t-1)$ is
\begin{equation}\label{Q function}
    Q_s(t)=\frac{\Gamma(s+n)\Gamma(s+1)}{2\Gamma(2s+n+1)t^{s+n}} {}_2F_1\left(s+n,s+1;2s+n+1;\frac1t\right),\qquad t>1.
\end{equation}
Here ${}_2F_1(a,b;c;z)=\sum_{j\ge0}(a)_j(b)_jz^j/((c)_j j!)$ is a hypergeometric function. This is one half of the kernel in \cite[Section~4.2]{Qiu}.

Writing $\rho=t-1$, the next logarithmic term in the expansion as $\rho\to0^+$ is $-\frac12 s(s+n)\rho\log\rho.$ At $s=0$ the kernel has the exact expression
\begin{equation}\label{Green constant at divisor}
    Q_0(1+\rho)=\frac12\left(\log\frac{1+\rho}{\rho} -\sum_{j=1}^{n-1}\frac1{j(1+\rho)^j}\right).
\end{equation}
In particular $Q_0(1+\rho)+\tfrac12\log\rho$ is smooth at $\rho=0$ and takes the value $-H_{n-1}/2$ there.

Let $\Gamma\subset G(\QQ)$ be the arithmetic group of a connected complex component and let $\Gamma_x$ be the stabilizer of $x$. For $\Re s$ sufficiently large define
\begin{equation*}
    m_{x,s}(z)=Q_s(1+\rho_x(z)),\qquad g_{x,s}(z)=\sum_{\gamma\in\Gamma_x\backslash\Gamma}m_{x,s}(\gamma z).
\end{equation*}
We use the finite-map multiplicities of $\Gamma_x\backslash D_x\to\Gamma\backslash D$ in this sum. At the CM point we also write
\begin{equation*}
    m_{v,s}(x)=m_{x,s}(z_0).
\end{equation*}
Then
\begin{equation*}
    m_{v,s}(x)=Q_s\left(1-\frac{q(x_2)}{q(x)}\right).
\end{equation*}

\begin{lemma}\label{Integral of Green function}
The series $g_{x,s}$ has a meromorphic continuation near $s=0$, with
\begin{equation*}
    \operatorname{Res}_{s=0}g_{x,s} =\frac{\deg_L(\Gamma_x\backslash D_x)}{2\deg_L(\Gamma\backslash D)}.
\end{equation*}
Its Laurent constant $\mathcal M_{Z(x)}=\widetilde{\lim}_{s\to0}g_{x,s}$ is weakly admissible and satisfies
\begin{equation}\label{Green mean}
    \int_{\Gamma\backslash D}\mathcal M_{Z(x)}\,\omega^n =-\frac{\deg_L(\Gamma_x\backslash D_x)}{2n}.
\end{equation}
Consequently the normalized admissible Green function is
\begin{equation}\label{admissible Green correction}
    \mathcal G_{Z(x)}=\mathcal M_{Z(x)} +\frac{\deg_L(\Gamma_x\backslash D_x)}{2n\deg_L(\Gamma\backslash D)}.
\end{equation}
\end{lemma}
\begin{proof}
The series converges absolutely and locally uniformly on compact subsets away from the divisor, for $\Re(s)$ sufficiently large; see \cite[Theorem~7.8.1]{OT} and \cite[Theorem~4.2.2]{Qiu}. With the leading singularity in \eqref{Q function}, the distributional form of \eqref{radial trace identity} is
\begin{equation*}
    \left(\frac{i}{2\pi}\partial\bar\partial g_{x,s} +\frac12\delta_{Z(x)}\right)\wedge\omega^{n-1} =\frac{s(s+n)}n g_{x,s}\,\omega^n.
\end{equation*}
Integration over the compact component gives
\begin{equation*}
    \int_{\Gamma\backslash D}g_{x,s}\,\omega^n =\frac{n\deg_L(\Gamma_x\backslash D_x)}{2s(s+n)} =\frac{\deg_L(\Gamma_x\backslash D_x)}{2s} -\frac{\deg_L(\Gamma_x\backslash D_x)}{2n}+O(s).
\end{equation*}
The resolvent is regular at this parameter on the orthogonal complement of the constant functions. Hence the only pole comes from the constant eigenspace, which gives the stated residue and the meromorphic continuation near $s=0$. The constant term of the Laurent expansion then defines a weakly admissible Green function and gives \eqref{Green mean}. Finally, subtracting its mean gives the admissible Green function and proves \eqref{admissible Green correction}.
\end{proof}

\begin{remark}
For $n=1$, \cite[Lemma~4.2]{Yuan1} gives
\begin{equation*}
    \int g_{\mathrm{Yuan}}\,c_1(L_B)=-1.
\end{equation*}
Since $c_1(L_B)=2c_1(L)$, its mean with respect to $c_1(L)$ is $-1/2$, in agreement with \eqref{Green mean}. The arithmetic pairing sums our Green function over all $2[F:\QQ]$ complex embeddings of $E$.
\end{remark}

\subsubsection*{Green function for generating series}
We apply the preceding construction to each weighted divisor $Z_t(\Phi)$, with $t\in F_+$. For $\Re s$ sufficiently large its Green function on a component is
\begin{equation*}
    \mathcal M_{Z_t,s}(z,h)=\sum_{y\in V(v)_t}\Phi_f(h^{-1}y)m_{y,s}(z).
\end{equation*}
For a fixed $t$, the containing vectors lie in the positive-definite space $W^\perp$ with fixed norm and lattice support, and hence form a finite set. We subtract their kernels before restricting to the zero-dimensional Shimura variety. The regularized constant of $Q_s(1+\rho)$ at $\rho=0$ is $\frac12(2\psi(1)-\psi(s+1)-\psi(s+n))$, where $\psi=\Gamma^{\prime}/\Gamma$. Since $n\ge1$, it is holomorphic near $s=0$; hence the subtraction leaves the residue in Lemma~\ref{Integral of Green function} unchanged. We continue the remaining expression meromorphically and take its Laurent constant coefficientwise. At the CM point the proper branches give
\begin{equation}\label{M series}
    \mathcal M_\Phi^{(v)}(g,\tau) =\widetilde{\lim}_{s\to0} \sum_{\substack{y\in V(v)-W^\perp\\q(y)\in F_+}} m_{v,s}(y)r(g,\tau)\Phi(y).
\end{equation}
The containing branches are computed by Lemma~\ref{Vanish of improper intersection}. With the normalized measure
\begin{equation*}
    \avint_{E^1\backslash E^1(\BA)}f(\tau)\,d\tau :=\frac1{2L(1,\eta)}\int_{E^1\backslash E^1(\BA)}f(\tau)\,d\tau,
\end{equation*}
the proper weak Green contribution at a pair of conjugate embeddings above $v$ is
\begin{equation*}
    2\avint_{E^1\backslash E^1(\BA)}\mathcal M_\Phi^{(v)}(g,\tau)\,d\tau.
\end{equation*}
This follows from the uniformization with the stabilizer weights mentioned in \eqref{CM component weights}.

\subsubsection*{The correction of Green function}
We now pass from the weak Green function to the normalized admissible one.

\begin{proposition}\label{extra term}
The change from weak to normalized admissible Green functions in the positive-index generating series satisfies
\begin{equation*}
    \bigl(\wh{\mathcal Z}^{\mathrm{adm}}_* -\wh{\mathcal Z}^{\mathrm{weak}}_*\bigr)\cdot\mathcal P =\bigl(\wh{\mathcal Z}^{\mathrm{adm}}_* -\wh{\mathcal Z}^{\mathrm{weak}}_*\bigr)\cdot\widehat\xi
    =-\frac{[F:\QQ]}n E_*(0,g,\Phi).
\end{equation*}
Thus this change cancels in the pairing with $\mathcal P-\widehat\xi$.
\end{proposition}

\begin{proof}
For each positive index, \eqref{admissible Green correction} adds $\deg_L Z_t/(2n\deg_L X)$ at each complex embedding. The component weights of $\mathcal P$ and $\widehat\xi$ agree by \eqref{CM component weights}. Summing over all embeddings and using \eqref{Geometric Siegel Weil} gives
\begin{equation*}
    \frac{[F:\QQ]}n\frac{\deg_L Z_*(g,\Phi)}{\deg_L X} =-\frac{[F:\QQ]}n E_*(0,g,\Phi).
\end{equation*}
Both positive pairings change by this number, which finishes the proof.
\end{proof}

\subsubsection*{Comparison at archimedean places}
We compare \eqref{M series} with the projected kernel of \cite[Theorem~3.3 and Lemma~4.1]{Guo1}. The factor $2$ below accounts for the conjugate pair of embeddings of $E$ above $v$.

\begin{proposition}\label{Comparison of archimedean K and M}
For the positive Fourier part, one has
\begin{equation*}
    2\overline{\mathcal K}_\Phi^{(v)}(g,\tau) -2\mathcal M_\Phi^{(v)}(g,\tau) =\left(\gamma+\log(4\pi)-2H_{n-1}-\frac1n\right) E_*(0,g,r(\tau)\Phi).
\end{equation*}
\end{proposition}
\begin{proof}
Fix a positive Fourier index and put $\rho=-q(y_2)/q(y)>0$. With the normalization in \cite[Lemma~4.1]{Guo1}, the kernels are
\begin{equation*}
    Q_s(1+\rho)=\frac12\int_0^1 \frac{(1-t)^s t^{n+s-1}}{(\rho+t)^{n+s}}\,dt, \qquad k_{v,s}(y)=\frac{\Gamma(n+s)}{2(4\pi)^s\Gamma(n)} \int_0^1\frac{t^{n+s-1}}{(\rho+t)^{n+s}}\,dt.
\end{equation*}
Their leading terms at infinity give
\begin{equation}\label{correct archimedean ratio}
    k_{v,s}(y)= \frac{\Gamma(n+1+2s)}{(4\pi)^s\Gamma(n)\Gamma(1+s)(n+s)} Q_s(1+\rho)+O_\epsilon(|s|\rho^{-n-1+\epsilon}).
\end{equation}
For $|s|\le\epsilon<1/2$ and $\rho\ge1$, the uniform estimate is
\begin{equation*}
    \left|k_{v,s}(y)- \frac{\Gamma(n+1+2s)}{(4\pi)^s\Gamma(n)\Gamma(1+s)(n+s)} Q_s(1+\rho)\right| \le C_\epsilon|s|\rho^{-n-1+\epsilon}.
\end{equation*}
Indeed, subtracting the common leading term leaves an additional factor $t/\rho$. The remaining difference is bounded by $C_\epsilon|s|(1-t)^{-\epsilon}(1+|\log(1-t)|)$, which is integrable. For a fixed positive index, the extra decay gives a normally convergent Poincar\'e sum near $s=0$; compare \cite[Lemma~6.1.10]{Qiu}. Its value at $s=0$ is zero, so its Laurent constant is zero.

The coefficient of $Q_s$ in \eqref{correct archimedean ratio} has value $1$ and logarithmic derivative
\begin{equation*}
    2H_{n-1}+\frac1n-\gamma-\log(4\pi)
\end{equation*}
at zero. By Lemma~\ref{Integral of Green function} and \eqref{Geometric Siegel Weil}, the residue of the corresponding weighted Green series is
\begin{equation*}
    -\frac12E_*(0,g,r(\tau)\Phi).
\end{equation*}
This prove the proposition by multiplying this residue by the logarithmic derivative.
\end{proof}

\subsection{The i-part at non-archimedean places}\label{The i-part at non-archimedean places}

At a finite place, we compute the intersection on the corresponding local Rapoport--Zink space. Suppose first that $E_v/F_v$ is unramified. We use the comparison of \cite[Theorem~5.12 and Corollary~5.14]{LMZ}, which also applies when $F_v/\QQ_p$ is ramified. Under this comparison, the selected Lie line and the polarization maps are preserved by \cite[Proposition~7.8]{LMZ}, while \cite[Proposition~8.13]{LMZ} identifies the lifting loci defined by the same special quasi-homomorphism. Thus the CM point, its marked special homomorphism, and the local intersection problem are unchanged under the comparison.

Now suppose that $E_v/F_v$ is ramified. By Assumption~\ref{Assumption}, the residue characteristic is odd and $F_v/\QQ_p$ is unramified. We then use the calculation of \cite[Section~6.1.2]{Qiu} on the ramified Rapoport--Zink spaces of \cite{RSZ1}. The same calculation applies to both the $\varpi_{E_v}$-modular and the almost $\varpi_{E_v}$-modular cases.

The explicit multiplicity below comes from the rank-two deformation factor containing the CM line.  This is why the same formula occurs in every ambient dimension.

At $g_v=1$, the proper local intersections are the terms indexed by $y\notin W^\perp$. For $y\in W^\perp\backslash\{0\}$, the corresponding divisor contains the CM cycle, and Lemma~\ref{Vanish of improper intersection} already includes the finite intersections with the residual divisor in its contribution. For $y=y_1+y_2$ with $y_2\ne0$, put
\begin{equation}\label{finite difference function}
    \bar k_{\Phi_v}(1,y) :=2k_{\Phi_v}(1,y)-m_{\Phi_v}(y)\log N_v.
\end{equation}
We will extend this difference to a Schwartz function on the nearby local space. The multiplicity $m_{\Phi_v}(y)$ is defined here for $y_2\ne0$.

We index the local terms by places of $F$. For a nonsplit place $v$, let $\nu$ be the place of $E$ above it; for a split place write $v=\nu_1\nu_2$. The normalized intersections are independent of enlarging the common RSZ reflex field, as in \eqref{normalized RSZ intersection}.

\subsubsection*{Basic relative setup}
Let $v$ be nonsplit in $E/F$, and let $\breve E_\nu$ be the completion of a maximal unramified extension of $E_\nu$.  Let $\BE$ be the one-dimensional supersingular strict formal $\mathcal O_{F_v}$-module of relative height two, endowed with an $\mathcal O_{E_v}$-action of signature $(1,0)$.  Equivalently, after choosing an embedding
\begin{equation*}
    \alpha:E_v\hookrightarrow B,
\end{equation*}
where $B/F_v$ is the division quaternion algebra, the image of $\mathcal O_{E_v}$ lies in the maximal order $\End_{\mathcal O_{F_v}}(\BE)$.  Let $\overline{\BE}$ denote the same strict $\mathcal O_{F_v}$-module with conjugate $\mathcal O_{E_v}$-action.

Write $X^{\vee_F}$ for the Faltings dual of a strict $\mathcal O_{F_v}$-module $X$.  Fix an $\mathcal O_{F_v}$-relative principal polarization
\begin{equation*}
    \lambda_{\BE}:\BE\longrightarrow\BE^{\vee_F}
\end{equation*}
whose Rosati involution induces the non-trivial automorphism of $E_v/F_v$.  Let $\mathcal E$ and $\overline{\mathcal E}$ be the canonical lifts of $\BE$ and $\overline{\BE}$ as strict $\mathcal O_{F_v}$-modules with $\mathcal O_{E_v}$-action. Faltings duality defines the lifted relative polarization $\lambda_{\mathcal E}$.  The computation of canonical lifting in \cite[Proposition~3.3]{Gro1} shows that for every $m\geq1$
\begin{equation*}
    \End_{\mathcal O_{F_v}} \big(\mathcal E\!\!\pmod{\varpi_{E_v}^{m+1}}\big) =\mathcal O_{E_v}+\varpi_{E_v}^{m} \End_{\mathcal O_{F_v}} \big(\mathcal E\!\!\pmod{\varpi_{E_v}}\big).
\end{equation*}

\subsubsection*{Inert case}
Assume that $v$ is inert in $E/F$, and let $V(v)$ be the nearby coherent Hermitian space at $v$.  We allow $F_v/\QQ_p$ to be ramified, using the relative comparison above.  Let $\mathcal N_n$ be the relative unitary Rapoport--Zink space of signature $(n,1)$ attached to the chosen self-dual level.  Its framing object is
\begin{equation*}
    \BX_n=\overline{\BE}\times\BE^n
\end{equation*}
with product relative polarization $\lambda_n$.  The functor parametrizes strict $\mathcal O_{F_v}$-modules with $\mathcal O_{E_v}$-action, relative polarization, and framing quasi-isogeny.  It is formally smooth over $\Spf\mathcal O_{\breve E_\nu}$ of relative dimension $n$.  The special-homomorphism space
\begin{equation*}
    \Hom_{\mathcal O_{E_v}}(\BE,\BX_n)_\QQ
\end{equation*}
is identified with $V(v)(E_v)$ and carries the Hermitian form
\begin{equation*}
    \langle x,y\rangle =\lambda_{\BE}^{-1}\circ y^{\vee_F}\circ\lambda_n\circ x \in\End_{\mathcal O_{E_v}}(\BE)_\QQ\cong E_v.
\end{equation*}
The group $\U(V(v)(E_v))$ acts on $\mathcal N_n$ through self-quasi-isogenies of the framing object preserving the relative polarization.  For $0\neq y\in V(v)(E_v)$, the Kudla--Rapoport divisor $\mathcal Z(y)$ is the locus on which the special quasi-homomorphism $y:\BE\dashrightarrow\BX_n$ lifts to a homomorphism from the canonical lift $\mathcal E$ to the universal strict module; compare \cite[Section~9.2]{MZ}.

Let $\widehat{\mathcal M}^{\mathrm b}_{\wt\nu,0}$ be the completion of a fixed toric component of the RSZ model along the basic locus.  The basic uniformization theorem gives
\begin{equation*}
    \widehat{\mathcal M}^{\mathrm b}_{\wt\nu,0} \cong \U(V(v))(F)\backslash \big(\mathcal N_n\times\U(\BV_f^v)/U^v\big),
\end{equation*}
followed by the absolute--relative comparison of \cite[Corollary~5.14 and Theorem~6.8]{LMZ}. Under this uniformization, the relative height-two factor gives a closed immersion $\mathcal N_0\hookrightarrow\mathcal N_n$, and
\begin{equation*}
    \wt{\mathcal P}_{\wt\nu} =\frac{1}{d_{\BW,U}} \U(W)(F)\backslash \big(\mathcal N_0\times\U(\BW_f^v)/U_\BW^v\big).
\end{equation*}
Under this uniformization the Kudla--Rapoport divisor satisfies
\begin{equation*}
    \wt{\mathcal Z}_t(\Phi)\big|_{\widehat{\mathcal M}^{\mathrm b}_{\wt\nu,0}} =\sum_{y\in\U(V(v))(F)\backslash V(v)_t} \sum_{h\in\U_y(\BA_f^v)\backslash\U(\BV_f^v)/U^v} \Phi^v(h^{-1}y)[\mathcal Z(y),h].
\end{equation*}

Consequently, writing $i_\nu$ for the numerical $\nu$-contribution to \eqref{DP term}, one has
\begin{equation*}
    i_\nu(Z_t(\Phi)) =\frac{\log N_v}{2L(1,\eta)} \int_{E^1\backslash E^1(\BA)} \sum_{y\in V(v)_t-W^\perp} m_{\Phi_v}(\tau^{-1}y)r(\tau)\Phi^v(y)\,d\tau.
\end{equation*}
Here the weighted multiplicity is the derived Euler--Poincar\'e characteristic
\begin{equation*}
    \begin{aligned}
        m_{\Phi_v}(\tau^{-1}y) &=m_{r(\tau)\Phi_v}(y)\\
        &=2\,\chi\!\left( \mathcal N_n, \mathcal O_{\mathcal Z(\tau^{-1}y)} \overset{\mathbf L}{\otimes}_{\mathcal O_{\mathcal N_n}} \mathcal O_{\mathcal N_0} \right).
    \end{aligned}
\end{equation*}
The factor $2$ comes from $\log N_\nu=2\log N_v$ at an inert place.  Since $E^1_v=\U(W_v)$ stabilizes $\mathcal N_0$, the multiplicity is invariant under $y\mapsto\tau y$.  For the transforms specified in Section~\ref{Proof of the main theorem}, put
\begin{equation*}
    m_{r(\tau)\Phi_v}(g,y):=m_{r(g,\tau)\Phi_v}(y).
\end{equation*}
The corresponding series of multiplicities is
\begin{equation*}
    \mathcal M_\Phi^{(v)}(g,\tau) =\sum_{\substack{y\in V(v)-W^\perp\\q(y)\in F_+}} m_{r(\tau)\Phi_v}(g,y)\,r(g,\tau)\Phi^v(y).
\end{equation*}
To obtain $i_\nu(Z_t(\Phi))$, multiply the $t$-th coefficient by $\log N_v$ and average over the zero-dimensional Shimura variety.

The relative height-two deformation calculation gives, for $y=y_1+y_2$ with $y_1\in W_v^\perp$ and $y_2\in W_v$,
\begin{equation*}
    m_{\Phi_v}(y) =(v(q(y_2))+1) 1_{\Lambda_v^\perp}(y_1) 1_{\mathcal O_{F_v}}(q(y_2)).
\end{equation*}
The relative height-two factor gives the same multiplicity as in \cite[Proposition~8.7]{YZZ2}.

\begin{proposition}\label{Unramified case}
Suppose that $v$ is inert in $E/F$ and unramified over $\QQ_p$, with the standard self-dual datum of \cite[Section~4.1]{Guo1}. Then
\begin{equation*}
    2k_{\Phi_v}(1,y)=m_{\Phi_v}(y)\log N_v \quad(y_2\ne0),\qquad \bar k_{\Phi_v}(1,\cdot)=0.
\end{equation*}
\end{proposition}
\begin{proof}
The relative canonical-lifting calculation above gives the formula for $m_{\Phi_v}$. It agrees with twice the analytic kernel in \cite[Lemma~4.2(1)]{Guo1}. Both functions have the same lattice support, proving the equality and the zero extension.
\end{proof}

\begin{proposition}
Suppose that $v$ is inert in $E/F$ and $F_v/\QQ_p$ is ramified, with the standard self-dual datum. The function \eqref{finite difference function} has a Schwartz extension to $V(v)_v$, with
\begin{equation*}
    \bar k_{\Phi_v}(1,0) =\frac{2(|d_v|-1)}{(1-N_v)(1+N_v^{-1})} \Phi_v(0)\log N_v.
\end{equation*}
\end{proposition}
\begin{proof}
The biformal comparison of \cite[Theorem~5.12, Corollary~5.14 and Proposition~8.13]{LMZ} preserves the canonical zero-dimensional Shimura variety and the lifting subscheme of its marked homomorphism. The normalized multiplicity is therefore given by the relative canonical-lifting calculation above. Subtracting it from twice the analytic kernel, \cite[Lemma~4.2(1)]{Guo1} gives the Schwartz extension and its displayed value at zero.
\end{proof}

\subsubsection*{Ramified case}

Assume that $v$ is ramified in $E/F$.  Then $v\nmid2$ and $F_v/\QQ_p$ is unramified by Assumption~\ref{Assumption}.  We use the notation of \cite[Section~6.1.2]{Qiu}.  The lattice $\Lambda_v$ is $\varpi_{E_v}$-modular when $n$ is odd and almost $\varpi_{E_v}$-modular when $n$ is even.

Assume $\varpi_{E_v}^2=\varpi_{F_v}$.  Let $\mathcal N_1$ be the ramified unitary Rapoport--Zink space with framing object
\begin{equation*}
    \BX_1=\mathcal O_{E_v}\otimes_{\mathcal O_{F_v}}\overline{\BE}.
\end{equation*}
If $n$ is odd, the framing object for $\mathcal N_n$ is
\begin{equation*}
    \BX_n=\BX_1\times(\BE^2)^{\frac{n-1}{2}},
\end{equation*}
where each factor $\BE^2$ is endowed with the polarization
\begin{equation*}
    \lambda= \matrixx{0}{\lambda_{\BE}\alpha(\varpi_{E_v})} {-\lambda_{\BE}\alpha(\varpi_{E_v})}{0}.
\end{equation*}
If $n$ is even, then
\begin{equation*}
    \BX_n=\BX_1\times(\BE^2)^{\frac{n-2}{2}}\times\BE,
\end{equation*}
and the polarization on the last factor is chosen so that the induced Hermitian line has determinant $q(e_v)$.  In both cases this gives the morphism
\begin{equation*}
    \mathcal N_1\longrightarrow\mathcal N_n
\end{equation*}
used in \cite[Section~6.1.2]{Qiu}.  We also have
\begin{equation*}
    \mathcal N_1=\mathcal N_1^+\coprod\mathcal N_1^-
\end{equation*}
by \cite[Proposition~6.4 and Example~12.2]{RSZ1}.

The morphism from the CM deformation space to $\mathcal N_1$ is the one constructed in \cite[Section~12]{RSZ1}.  With the element $c\in\mathcal O_B^\times$ fixed in Steps 4--5 of \cite[Section~6.1.2]{Qiu}, the local CM cycle is
\begin{equation*}
    \wt{\mathcal P}_{\wt\nu} =\frac{1}{2d_{\BW,U}} \U(W)(F)\backslash \big(\mathcal N_0^{c,\pm}\times\U(\BW_f^v)/U_\BW^v\big).
\end{equation*}
The factor $2$ is the number of components of $\mathcal N_0^{c,\pm}$.

We use the compatible local lattice notation of \cite[Section~4.1]{Guo1}:
\begin{equation*}
    \Lambda_v^\perp=\mathcal O_{E_v}e'_v\oplus\Lambda_{1,v}^\perp, \qquad \Phi_v|_{W_v^\perp} =1_{u\mathcal O_{E_v}e'_v\oplus\Lambda_{1,v}^\perp}.
\end{equation*}
The second equality follows by intersecting the modular rank-two lattice with its positive line. We also use the quaternionic parametrization of \cite[Section~6.1.2]{Qiu} and compute the proper length on each Serre branch from \cite[Theorem~1.4]{Wew}, with the Hermitian form induced by the original polarizations.

Write $u^2=\varpi_{F_v}$, and let $T$ denote the action of $u$ on $\mathbb E$. By the Serre tensor description of \cite[Section~6.1.2]{Qiu}, a special quasi-homomorphism is determined by
\begin{equation*}
    b\in B=\operatorname{End}_{\mathcal O_{F_v}}^0(\mathbb E)
\end{equation*}
and has the form
\begin{equation*}
    j_b=\binom{bT}{b}, \qquad T^2=\varpi_{F_v}.
\end{equation*}
The polarization gives
\begin{equation*}
    q(j_b)=-2\varpi_{F_v}\operatorname{Nrd}(b);
\end{equation*}
see \cite[(3.5)--(3.6) and Lemma~3.5]{RSZ3}.

The element $c\in\mathcal O_B^\times$ defining the CM branch identifies its canonical lifting with $\mathbb E$. Using the CM embedding $E_v\subset B$, we can decompose $b$ into its CM part and its orthogonal part. Write $b_-\in E_v$ for the coefficient of the latter with respect to an integral generator. The CM part lifts automatically along the CM cycle, while the second part gives the $W_v$-component $y_2$. Since the two parts are orthogonal for the norm form, for $y_2\ne0$ we have
\begin{equation*}
    r:=v(q(y_2)) =1+v_{E_v}(b_-), \qquad v_{E_v}(u)=1.
\end{equation*}

For the canonical lifting $\mathcal E$ over $R=\mathcal O_{\breve E_v}$, \cite[Theorem~1.4]{Wew} gives
\begin{equation*}
    \operatorname{End}_{\mathcal O_{F_v}}(\mathcal E\bmod u^k) =\mathcal O_{E_v}+u^{k-1}\mathcal O_B, \qquad k\ge1.
\end{equation*}
Here the exponent $k-1$ corresponds to reduction modulo $u^k$. The successive quotients have length one over $\mathcal O_{E_v}$; see \cite[Lemmas~1.3,~1.5 and Theorem~1.4]{Wew}. Consequently $j_b$ lifts modulo $u^k$ exactly when
\begin{equation*}
    b_+\in\mathcal O_{E_v},\qquad b_-\in u^{k-1}\mathcal O_{E_v}.
\end{equation*}
Necessity follows from the second coordinate of $j_b$; sufficiency follows by composing its lift with the source action $T$. Thus the local intersection multiplicity is $r$ for $r\ge1$, and is zero for $r\le0$. In particular, when $r=0$ there is no intersection on the special fibre, while $r=1$ means that the homomorphism exists modulo $u$ but does not lift modulo $u^2$.

The rank-one positive factor contributes exactly on $u\mathcal O_{E_v}e'v,$ while the remaining positive factors contribute the support $\Lambda_{1,v}^\perp$. The two Serre branches have the same intersection multiplicity $r$. Their two equal contributions cancel the factor $1/2$ in the normalization of the zero-dimensional Shimura variety. Hence, for both parities of $n$,
\begin{equation}\label{ramified proper multiplicity}
    m_{\Phi_v}(y_1+y_2) =r\,\Phi_v(y_1)1_{\Lambda_{W_v}}(y_2), \qquad r=v(q(y_2)).
\end{equation}
The expression is zero outside the indicated support.

The normalized finite intersection is consequently
\begin{equation*}
    i_v(Z_t(\Phi)) =\frac{\log N_v}{2L(1,\eta)} \int_{E^1\backslash E^1(\BA)} \sum_{y\in V(v)_t-W^\perp} m_{\Phi_v}(\tau^{-1}y)r(\tau)\Phi^v(y)\,d\tau.
\end{equation*}

\begin{proposition}
Let $v$ be ramified in $E/F$, with the standard datum of \cite[Section~4.1]{Guo1}. The Schwartz extension of \eqref{finite difference function} is
\begin{equation}\label{ramified difference function}
    \bar k_{\Phi_v}(1,y_1+y_2) =N_v^{-1}1_{\Lambda_v^\perp}(y_1) 1_{\Lambda_{W_v}}(y_2)\log N_v.
\end{equation}
In particular, $\bar k_{\Phi_v}(1,0)=N_v^{-1}\log N_v$.
\end{proposition}
\begin{proof}
By \cite[Lemma~4.2(2)]{Guo1}, for $y_2\ne0$,
\begin{equation*}
    2k_{\Phi_v}(1,y_1+y_2) =\bigl(v(q(y_2))\Phi_v(y_1) +N_v^{-1}1_{\Lambda_v^\perp}(y_1)\bigr) 1_{\Lambda_{W_v}}(y_2)\log N_v.
\end{equation*}
Subtracting \eqref{ramified proper multiplicity} times $\log N_v$ gives \eqref{ramified difference function}, which is a Schwartz function on the whole local space.
\end{proof}

\subsubsection*{Split case}
Let $v=\nu_1\nu_2$ split in $E$. We use the two relative EL factors on the fixed zero-dimensional Shimura variety and compute on the same branches of the fibre product as above.

\begin{proposition}\label{Ordinary case}
For the standard maximal-lattice datum, the proper split-place contribution is zero.
\end{proposition}

\begin{proof}
We first describe the special divisor on the standard split EL chart. The relevant $p$-divisible group is an extension
\begin{equation*}
    0\longrightarrow G\longrightarrow X\longrightarrow E^n\longrightarrow0,
\end{equation*}
where $E$ is the \'etale height-one strict $\mathcal O_{F_v}$-module and $G$ is the height-one formal module. Choose Serre--Tate coordinates $z_1,\ldots,z_n\in G$ for this extension. If a positive special homomorphism is represented by $(a_1,\ldots,a_n)$, then it lifts precisely when
\begin{equation*}
    [a_1]_G(z_1)+_G\cdots+_G[a_n]_G(z_n)=0.
\end{equation*}
Since $q(x)\ne0$, the coefficients $a_i$ are not all zero. Let $\pi^r$ be their common content. After an integral change of basis, the lifting equation becomes
\begin{equation*}
    [\pi^r]_G(z_1)=0.
\end{equation*}
The power series $[\pi^r]_G(z_1)$ is distinguished of degree $N_v^r$. Hence its quotient ring is finite free over
\begin{equation*}
    \mathcal O_{\breve F_v}[[z_2,\ldots,z_n]],
\end{equation*}
so the special divisor is flat and is the closure of its generic fibre. Also, along the component $z_1=0$, the remaining factor satisfies
\begin{equation*}
    \left.\frac{[\pi^r]_G(z_1)}{z_1}\right|_{z_1=0} =\pi^r\cdot\mathrm{unit}.
\end{equation*}
Thus the finite residual multiplicity is $r$, in agreement with the comparison in Lemma~\ref{Vanish of improper intersection}.

We now show that no proper special divisor meets the CM cycle. At a point of the fixed CM cycle, write the corresponding special homomorphism according to
\begin{equation*}
    V_v=W_v^\perp\oplus W_v.
\end{equation*}
Its $W_v$-component would give a homomorphism from the relevant \'etale factor of $\mathcal A_0$ to the multiplicative factor corresponding to $W_v$. These two $p$-divisible groups have different slopes, so such a homomorphism must be zero. Therefore any special homomorphism occurring at an intersection point has zero $W_v$-component. Equivalently,
\begin{equation*}
    Z(y)\cap\mathcal P\ne\varnothing \quad\Longrightarrow\quad y\in W^\perp.
\end{equation*}
The proper terms are indexed by $y\notin W^\perp$. Hence they are disjoint from the CM cycle, and the proper split-place contribution is zero at the standard maximal level.

Finally, if the divisor is pulled back to a smaller level, \eqref{CM-level-pairing} shows that its normalized pairing with the fixed CM cycle is unchanged. The same vanishing therefore holds at the levels considered here.
\end{proof}

This gives a direct geometric explanation for the absence of a proper contribution at split places. At a split place, a point of the CM cycle can lie on a special divisor only when the corresponding special homomorphism has zero $W_v$-component, hence only when $y\in W^\perp$. This is the same phenomenon behind the split-place calculation in \cite[Section~8.4]{YZZ2}: only the degenerate $W^\perp$-part remains, while the proper contribution vanishes.

\subsection{Proof of the main theorem}\label{Proof of the main theorem}
We now combine all the explicit results. We first take $g_f=1$ and retain all positive Fourier indices. This gives the explicit expression in Theorem~\ref{main theorem of arithmetic Siegel-Weil I}.

The holomorphic projection of the derivative is decomposed in \cite[Theorem~3.3]{Guo1}. We first combine its proper local terms with the arithmetic intersections, and then treat the  terms from improper intersections.

For the proper branches, the difference of the analytic and arithmetic kernels is
\begin{equation*}
    \begin{aligned}
        &-\sum_{v\mid\infty}\avint_{E^1\backslash E^1(\BA)} \bigl(2\overline{\mathcal K}^{(v)}_\Phi(g,\tau) -2\mathcal M^{(v)}_\Phi(g,\tau)\bigr)\,d\tau\\
        &-\sum_{\substack{v\nmid\infty\\v\ \mathrm{nonsplit}}} \avint_{E^1\backslash E^1(\BA)} \sum_{\substack{y\in V(v)-W^\perp\\q(y)\in F_+}} \bar k_{r(\tau)\Phi_v}(1,y)r(g,\tau)\Phi^v(y)\,d\tau
        -\frac{[F:\QQ]}n E_*(0,g,\Phi).
    \end{aligned}
\end{equation*}
Propositions~\ref{extra term} and~\ref{Comparison of archimedean K and M} combine the first and last terms into
\begin{equation*}
    -[F:\QQ]\bigl(\gamma+\log(4\pi)-2H_{n-1}\bigr)E_*(0,g,\Phi).
\end{equation*}
Proposition~\ref{Ordinary case} gives zero for the proper split terms.

For a containing vector $y\in W^\perp,y\ne0$, the term $+\log|q(y)|_f$ in the projected derivative cancels the term $-\log|q(y)|_f$ in Lemma~\ref{Vanish of improper intersection}. Their remaining contribution is a constant times $\Phi|_{W^\perp}$.

Complete each proper nonsplit sum to the positive-index theta series on $V(v)$ using the Schwartz extensions of Section~\ref{The i-part at non-archimedean places}. It is not hard to check that the extra boundary is the restriction of $\bar k_{\Phi_v}$ to $W_v^\perp$. Together with the analytic containing terms, the local inputs left on $W^\perp$ are finite linear combinations of
\begin{equation*}
    \Phi_v|_{W_v^\perp},\qquad \bar k_{\Phi_v}|_{W_v^\perp},\qquad c_{\Phi_v}(1,\cdot).
\end{equation*}
The first two are Schwartz. For the third, a finite tensor decomposition $\Phi_v=\sum_j\phi_{1,j}\otimes\phi_{2,j}$ in the definition of \cite[(3.1.5)]{Guo1} gives
\begin{equation*}
    c_{\Phi_v}(1,y_1) =\sum_j\phi_{1,j}(y_1) {W_{0,v}^{\circ}}'(0,1,\phi_{2,j}),
\end{equation*}
which is also Schwartz. At standard unramified places with $|d_v|=1$, both the proper difference and this boundary vanish. Thus the remaining terms form a finite non-singular pseudo-theta sum $A(g,\Phi)$ with $V_1=W^\perp\subsetneq V$ and $V_0=\{0\}$. The Schwartz bounds and the fixed archimedean Gaussians give absolute convergence. We have proved
\begin{equation*}
    \begin{aligned}
        &\Pr'I'(0,g,\Phi)+\wh{\mathcal Z}_*(g,\Phi)\cdot\mathcal P\\
        ={}&-[F:\QQ]\bigl(\gamma+\log(4\pi)-2H_{n-1}\bigr)E_*(0,g,\Phi)\\
        &-\sum_{\substack{v\nmid\infty\\v\ \mathrm{nonsplit}}} \avint_{E^1\backslash E^1(\BA)} \theta_*(g,\tau,\bar k_{\Phi_v}\otimes\Phi^v)\,d\tau+A(g,\Phi).
    \end{aligned}
\end{equation*}
Here $\bar k_{\Phi_v}$ is the Schwartz extension at $g_v=1$, $\theta_*$ is its positive Fourier part, and $A(g,\Phi)$ is the finite degenerate sum just described. Proposition~\ref{Unramified case} makes this extension zero at standard unramified places with $|d_v|=1$. The sum is therefore supported at the places ramified in $E/F$, the places dividing the different of $F/\QQ$, and the fixed nonstandard Schwartz places.

We have proved the comparison for $g_f=1$. We now check that the same formula is compatible with the finite Weil action. For an upper unipotent element, both sides are multiplied term by term by $\psi_v(bq(y))$, so there is nothing further to prove. For a diagonal element, we have
\begin{equation*}
    r(m(a_v))\Phi_v(y) =\chi_{\BV,v}(a_v)|a_v|_{E_v}^{(n+1)/2}\Phi_v(a_vy).
\end{equation*}
The same scalar appears in the weighted divisor by \eqref{invariance property of weighted divisor}. We also need to keep track of the change of the local divisor equation when $y$ is replaced by $a_vy$. This is exactly the scaling relation \eqref{obstruction scaling}. After summing the valuations at conjugate places as in \eqref{conjugate-valuations}, the resulting logarithmic term agrees with the term $2\log\delta_f(g_f)$ in \cite[Theorem~3.3]{Guo1}. Thus the arithmetic and analytic sides have the same transformation under the upper-triangular Weil action.

At the standard unramified places, the usual compact invariance together with the Iwasawa decomposition gives the same conclusion for the required local elements. At the remaining places, we choose a sufficiently small compact open subgroup which fixes $\Phi_v$, the Schwartz functions appearing above, and the local equations used in the intersection calculation. We also take it inside the maximal compact subgroup on which $\delta_v$ is unchanged. Hence, for every element $k_v$ of this subgroup,
\begin{equation*}
    r(k_v)\Phi_v=\Phi_v,\qquad \delta_v(hk_v)=\delta_v(h).
\end{equation*}
The Whittaker terms and the geometric sums are then unchanged. Therefore the identity proved at $g_f=1$ also holds on the compact neighbourhoods stated in the theorem.

Finally, for later use, we record the remaining contribution supported on $W^\perp$. Combining the containing terms above with the boundary values of $\bar k_{\Phi_v}$ and $c_{\Phi_v}$ gives
\begin{equation*}
    -2\frac{L_f'(0,\eta)}{L_f(0,\eta)} -\LL\cdot\mathcal P.
\end{equation*}
For the standard datum this is zero by Theorem~\ref{Modular height of CM point}. More generally, if the finite data are changed only at a finite set $S$ of places, the additional local terms are algebraic multiples of $\log N_v$ for $v\in S$. Thus the same comparison holds modulo $\overline{\QQ}\log S$. We will use this fact in \cite{Guo2}.

\noindent \small{School of mathematical sciences, Peking University, Beijing 100871, China}

\noindent \small{\it Email: ziqiguo0603@pku.edu.cn}

\end{document}